\documentclass[11pt]{article}
\usepackage{amsrefs, amsmath, amsfonts, amsthm, latexsym, comment}
\usepackage[hidelinks]{hyperref}
\usepackage[margin=0.85in]{geometry}
\usepackage{enumerate}
\usepackage{authblk}
\usepackage{tikz}
\usepackage{caption}
\usetikzlibrary{plotmarks}

\usepackage{xcolor}

\newenvironment{mitemize}
{ \begin{itemize}
    \setlength{\itemsep}{0pt}
    \setlength{\parskip}{0pt}
    \setlength{\parsep}{0pt}     }
{ \end{itemize}       }

\usepackage{enumitem}

\newtheorem{thm}{Theorem}
\newtheorem{thma}{Theorem}

\newtheorem{clm}{Claim}[section]
\newtheorem{obs}[clm]{Observation}
\newtheorem{lem}[clm]{Lemma}
\newtheorem{prop}[clm]{Proposition}
\newtheorem{defin}[clm]{Definition}
\newtheorem{cor}[clm]{Corollary}
\newtheorem{conj}{Conjecture}
\newtheorem{problem}{Problem}

\newcommand{\cS}{\mathcal{S}}

\newcommand{\cF}{\mathcal{F}}

\newcommand{\N}{\mathbb{N}}

\newcommand{\R}{\mathbb{R}}

\newcommand{\Pro}{\mathbb{P}}
\newcommand{\E}{\mathbb{E}}
\newcommand{\cN}{\mathcal{N}}

\newcommand{\ep}{\varepsilon}
\newcommand{\al}{\alpha}
\newcommand{\g}{\gamma}

\newcommand{\D}{\Delta}

\newcommand{\ind}{1{\hskip -2.5 pt}\hbox{I}}
\newcommand{\ov}[1]{\overline #1}

\makeatletter
\newcommand{\subjclass}[2][2010]{%
  \let\@oldtitle\@title%
  \gdef\@title{\@oldtitle\footnotetext{#1 \emph{Mathematics subject classification.} #2}}%
}
\newcommand{\keywords}[1]{%
  \let\@@oldtitle\@title%
  \gdef\@title{\@@oldtitle\footnotetext{\emph{Key words and phrases.} #1.}}%
}
\makeatother

\title{Convergence of the Quantile Admission Process with Veto Power}

\author{Naomi Feldheim\thanks{Weizmann Institute of Science, email: naomi.feldheim@weizmann.ac.il. Research partially supported by an NSF postdoctoral fellowship at Stanford, and by ISF grant 147/15.}
\hspace{2pt}  and Ohad Noy Feldheim\thanks{The Hebrew University of Jerusalem, email: ohad.feldheim@mail.huji.ac.il. Research partially supported by a postdoctoral fellowship in Stanford's mathematics department.}}

\date{}

\begin{document}
\maketitle
\begin{abstract}
The \emph{quantile admission process with veto power} is a stochastic processes suggested by Alon, Feldman, Mansour, Oren and Tennenholtz  as a model for the evolution of an exclusive social group.

The model itself consists of a growing multiset 
of real numbers, representing the opinions of the members of the club. On each round two new candidates, holding i.i.d. $\mu$-distributed opinions, apply for admission to the club. The one whose opinion is minimal is then admitted if the percentage of current members closer in their opinion to his is at least $r$. Otherwise neither of the candidates is admitted.

We show that for any $\mu$ and $r$, the empirical distribution of opinions in the club converges to a limit distribution. We further analyse this limit, show that it may be non-deterministic and provide conditions under which it is deterministic.

The results rely on a recent work of the authors relating tail probabilities of mean and maximum of any pair of unbounded i.i.d. random variables, and on a coupling of the evolution of the empirical $r$-quantile of the club with a random walk in a changing environment.
\end{abstract}

{\footnotesize
}
\section{Introduction}

Consider the following stochastic model for increasing sets $\cS=\{\cS_t\}_{t\in\N_0}$.
Let $\mu$ be an arbitrary probability distribution on $\R$, let $r\in(0,1)$ be  the \emph{quantile} parameter of the model, and let $(X_t,Y_t)_{t\in\N_0}$ be a collection of pairs of i.i.d. $\mu$-distributed random variables. For every $t\in\N_0$ denote the set $\cS_t^+:=\{s \in \cS_t\ :\ |s-\min(X_t,Y_t)|\le |s-\max(X_t,Y_t)|\}$.
Initialise the model with $\cS_0=\emptyset$
and for every time step $t\in\N$ set
\[
\cS_t=
\begin{cases}
\cS_{t-1}\cup\{\min(X_t,Y_t)\}, & |\cS^+_{t-1}|\ge r|\cS_{t-1}|\\
\cS_{t-1}, &\text{otherwise.}
\end{cases}
\]
In addition, define $\nu^*_t$, the empirical distribution of opinions at step $t$, by $\frac1{|\cS_t|}\sum_{s\in \cS_t}\delta_s$ where $\delta_s$
is Dirac's delta measure at $s$.

This model, called the \emph{quantile admission process with veto power}, was suggested by Alon, Feldman, Mansour, Oren and Tennenholtz in \cite{social} as a model for the evolution of certain exclusive social clubs, that is, clubs whose  present members take part in an admission procedure for screening and selecting new members.
In our particular scenario, the distinction between candidates is represented by a one-dimensional parameter in $\R$
(e.g. from political left to right).
The selection itself involves two elements, a \emph{voting system} and the founder's \emph{veto power}. At every time step two random candidates from the general population are considered for admission to the club. Each club member then votes for the candidate whose opinion is closest to his (abstaining in case of a draw). Finally, if a candidate holding the minimal opinion received at least $r$-fraction of the votes (including abstainers) then such a candidate is admitted. Otherwise, both candidates are rejected.

 Motivation for studying the model is further discussed in Section~\ref{sec:bg}. We remark that the fine details of this admission procedure (such as how to handle draws) appear to be immaterial for the results. The only important consideration in selecting these is avoiding the possibility of a deadlock under which no additional members will ever be admitted.

The authors of \cite{social} studied this model with $\mu$ which is uniform on $[0,1]$ for arbitrary $r\neq \frac12$. They showed that under these assumptions, $\nu^*_t$ converges to a deterministic distribution
and obtained bounds on its empirical $r$-quantile. In addition, they showed that this process demonstrates a counter-intuitive behaviour when $r$ is greater than $1/2$, when the empirical distribution of the club weakly converges to an atom at $1$, the supremum of opinions in the population.

Here we show that $\nu^*_t$ converges weakly for any distribution $\mu$, and analyse its form. In addition we provide necessary and sufficient conditions under which this limit distribution is deterministic, and provide several examples. In particular we recover the results of \cite{social} and extend them to the case $r= \frac12$, for which the limit is non-deterministic. Part of our interest in the process stems from its relations with \emph{random walks in changing environment} and from the new probabilistic tool-set that was used handle it. For a survey of this see Section~\ref{sec:tools and methods}.

\subsection{Results}
Throughout we denote by $X$ and $Y$ a pair of independent, $\mu$-distributed random variables.
Denote by $\mu_{\min}$ and $\mu_{\max}$ the infimum and supremum of the support of $\mu$, respectively, and write for
$m<\mu_{\max}$:
\vspace{-5pt}
\begin{align}\label{eq: rho}
\rho(m)=\rho_{\mu,r}(m) &:= \frac{\mu\big((m,\infty)\big)^2}{\mu*\mu\big([2m,\infty)\big)}-(1-r)
\\ &\phantom{:}= \Pro(X > m, Y>m \mid X+Y\ge 2m) - (1-r). \notag
\end{align}
Our main result is the following.

\begin{thm}\label{thm: main}
For any probability distribution $\mu$ on $\R_+$, the sequence $\nu^*_t$ converges weakly to a (possibly random) distribution $\nu_\infty$ on $\R_+$, almost surely.\\[-12pt]

Moreover, the limit distribution $\nu_\infty$ satisfies one of the following forms:\\[-18pt]
\begin{enumerate}[label={\emph(\roman*\emph)},ref={(\roman*)}]
\item \label{item:m criterion}
$\nu_\infty((s,\infty)) = \Pro\left( X>s,\, Y> s \mid X+Y\ge 2m\right)$
for $m<\mu_{\max}$ which satisfies
\vspace{-5pt}
\begin{equation*}
\rho(m)\le 0 \le \lim_{q\to m-}\rho(q) ,
\end{equation*}
\vspace{-25pt}
\item \label{item:m plus criterion}
$\nu_\infty((s,\infty)) = \Pro\left(X>s,\, Y > s \mid X+Y > 2m\right)$
for $m<\mu_{\max}$ which satisfies
\vspace{-5pt}
\begin{equation*}
\rho(m)\le 0 = \lim_{q\to m+}\rho(q),
\end{equation*}
\vspace{-25pt}
\item\label{item:max criterion}
 $\nu_{\infty}=\delta_{\mu_{\max}}$, and
$0\le\lim_{q\to\mu_{\max}}\rho(q).$
\end{enumerate}
\end{thm}

In case that $\mu$ is continuous, an explicit form of the possible limiting measures of the admission process $\nu_\infty$ is given by the following corollary.

\begin{cor}\label{cor: limit}
Assume that $\mu$ is a continuous measure on $\R_+$. Then, almost surely, either $\nu_\infty=\delta_{\mu_{\max}}$ or
$d\nu_\infty(x) \propto \mu\big([x,\infty)\cap [2m-x,\infty)\big) d\mu(x)$
where $m$ satisfies $\rho_{\mu,r}(m)=0$.
\end{cor}

Observe that the empirical distribution of the club members' opinions in the veto-power process demonstrates a lighter tail near infinity, in comparison with the
distribution of the general population (even though, as will become evident in Section \ref{sec: examples} below, the mean of the empirical distribution of the club members' opinions is often larger then the {mean of $\mu$}).
\begin{cor}\label{cor: tail}
If $\mu_{\max}=\infty$ and $\nu_\infty=\lim_{t\to\infty} \nu_t^*$ then, almost surely,
$\lim_{s\to\infty}\frac{\log \nu_\infty(s,\infty)}{\log \mu(s,\infty)} = 2$.
\end{cor}

However, $\nu_\infty$ need not be deterministic.

\begin{prop}\label{prop: atom exm}
For any $r\in(0,1)$, there exists a measure $\mu$ for which there are infinitely many possible limiting distributions for $\lim_{t\to\infty}\nu_t^*$, each occurring with positive probability.
\end{prop}

We further provide a necessary and sufficient condition for having a deterministic $\nu_\infty$ for all $r\in(0,1)$.
This extends \cite{social}*{Theorem 3.5} which provides a weaker sufficient condition.

\begin{thm}\label{thm: determinism}
Let $\mu$ be a given probability measure.
Then $\nu_\infty$ is deterministic for all $r\in (0,1)$ if and only if $\rho_{\mu,1}(m)$ is strictly monotone on the support of $\mu$.
\end{thm}

Finally we provide the following criterion for verifying the monotonicity of $\rho$.

\begin{prop}\label{prop: convex exm}
Let $g(x)=-\log\mu((x,\infty))$ for $x\in [\mu_{\min},\mu_{\max}]$. If $g$ is thrice differentiable with $g'' \ge 0$ and $g'''\le 0$ on $(\mu_{\min},\mu_{\max})$,
then $\rho_{\mu,1}(m)$ is strictly monotone decreasing.
\end{prop}

\subsection{Examples}\label{sec: examples}

Using these results, 
it is possible to compute the limit $\nu_\infty$ for many example of $\mu$ and $r$. We state here a few examples of interest, omitting calculations which are straightforward.
Throughout, we fix $r\in (0,1)$.

 \medskip
 {\bf Uniform distribution}
(extending \cite{social}*{Section 4}).
Let $\mu$ be the uniform distribution on $[0,1]$.
The function $\rho(m)$ from \eqref{eq: rho} is computed to be
\vspace{-3pt}
\[
\rho(m) = \begin{cases}
\frac{(1-m)^2}{1-2m^2}-(1-r), & 0\le m< \frac 1 2\\
\frac{1}{2} - (1-r), & \frac 1 2 \le m\le 1.
\end{cases}
\]
Using Theorem \ref{thm: main} we deduce different behaviour in the following cases. The first two cases recover the results of \cite{social}, whereas the case $r=\frac 1 2$ is new.
\begin{itemize}
\item $r>\frac 1 2$: a.s., the limit is $\nu_\infty =\delta_1$.
\item $r<\frac 1 2$: a.s., the limit $\nu_\infty$ is an absolutely continuous measure with density proportional to
    \vspace{-3pt}
\begin{equation*}
d\nu_\infty \propto
\begin{cases}
x-(2m-1), & 0\le x\le m,\\
1-x, & m\le x \le 1,
\end{cases}
\end{equation*}
where $m\in[0,\frac 1 2]$ is a solution to $\frac{(1-m)^2}{1-2m^2}=1-r$, namely $m=\frac{1-\sqrt{1-3r+2r^2}}{3-2r}$.

Observe that as $r\to\frac 1 2$, the density of the limiting measure $\nu_\infty$ approaches the centred triangle function $\min(x,1-x)$.
This may be viewed as a perfect balance between the bias caused by the veto power and the bias caused by the quorum-admission condition.

\item $r=\frac 1 2$: the limit $\nu_\infty$ is an absolutely continuous measure with density proportional to
    \vspace{-3pt}
\begin{equation*}
d\nu_\infty \propto
\begin{cases}
x-(2M-1), & 2M-1\le x\le M,\\
1-x, & M\le x \le 1,
\end{cases}
\end{equation*}
where $M$ is a random variable supported on $[\frac 1 2, 1]$.
Moreover, by Lemma \ref{lem: zero drift barrier} below, $\Pro(M \in I)>0$ for any $I\subseteq [\frac 1 2 ,1]$.
In particular, the limit is not deterministic (as can be verified by Theorem~\ref{thm: determinism}).

\end{itemize}

\medskip
{\bf Normal distribution.} Let $\mu\sim \cN(0,1)$. In this case $\frac 1 2 \mu * \mu \sim \cN(0,1)$, and we obtain from \eqref{eq: rho} by direct computation that
\[
\rho_{\mu,r}(m) = \frac{\mu\big( (m,\infty) \big)^2}{\mu\big([m,\infty)\big)} -(1-r) = \mu\big((m,\infty)\big)-(1-r).
\]
Therefore, for any $r\in [0,1]$ the function $m\mapsto \rho_{\mu,r}(m)$ is strictly monotone on $\R$, and we obtain by Theorem \ref{thm: determinism} that $\nu_\infty$ is unique. By Theorem \ref{thm: main}, the limiting measure $\nu_\infty$ has density proportional to
\[
d\nu_\infty \propto
\begin{cases}
e^{-x^2/2}\int_{2m-x}^\infty e^{-t^2/2} dt, & x\le m,\\
e^{-x^2/2}\int_{x}^\infty e^{-t^2/2} dt, & x>m,
\end{cases}
\]
where $m$ is the unique solution to $\frac 1 {\sqrt{2\pi}}\int_m^\infty e^{-t^2/2}= 1-r$.

 \medskip
{\bf Exponential distribution.} Let $\mu \sim \text{Exp}(1)$. Then $\nu^*_t$ will almost surely converge to the unique measure $\nu_\infty$ with density
\[
\frac{1 }{m+\frac 1 2}\cdot \begin{cases}
1, & x\le m, \\
e^{-2(x-m)}, & x>m,
\end{cases}
\]
where $m = \frac{r}{2(1-r)}$. Observe that when $r$ tends to 1, $m$ tends to infinity, demonstrating a drift to the right much like the uniform case.

\medskip
{\bf Compressed exponential distributions.} Consider the measure $\mu([x,\infty)) =e^{-x^\al}\ind(x\ge 0)$ for $\al\in [1,2]$. By Proposition \ref{prop: convex exm} and Theorem \ref{thm: determinism}, for all $r$ the limit $\nu_\infty$ is deterministic. By Corollary \ref{cor: limit} it has density which is proportional to
\vspace{-5pt}
\[
\al x^{\al-1} e^{-x^\al}\cdot
\begin{cases}
e^{-(2m-x)^{\al}}, & x\le m\\
e^{-x^{\al}}, & x>m,
\end{cases}
\]
\vspace{-12pt}

where $m$ is the unique solution to
\vspace{-5pt}
\[
e^{2m^\al} \left(\int_0^{2m} \al x^{\al-1} e^{-x^\al-(2m-x)^\al} dx + e^{-(2m)^\al}\right)  = \frac 1{1-r}.
\]
\vspace{-12pt}

For $\al=2$ this reduces to $\int_0^{2m} 2x e^{-2(x-m)^2 }dx  + e^{-2m^2} = \frac {1}{1-r}$,
 so that $m$ is of the order $\frac {1}{1-r}$.

%


\subsection{Background and motivation}\label{sec:bg}

The study of evolving social groups originates in sociology and economics, where analysis of such processes is expected to improve our understanding
of human social behaviour, allowing us to anticipate the implications of various factors involved in the evolution of clubs, corporations, organizations and societies.
At the same time, the models suggested for the evolution of social groups turned out to be of independent interest for the development of probability theory.

\medskip
\textbf{Sociological interest.} Most of the research conducted on evolving social groups revolves around \emph{decentralised} social groups, in which local
laws are used to govern the admission into the group. These include social networks, business relations and the set of individuals holding a certain opinion.
These social groups are characterised by the fact that admission of a new member into the group could be initiated by each of a certain class of members without any interaction with the rest of the members. The evolution of such groups is somewhat reminiscent of condensation behaviours
in statistical mechanics. Indeed, such models were recently given the title ``Statistical Physics of Social Dynamics'' by Castellano, Fortunato and Loreto~\cite{CFL}. In particular, a few models for opinion formation and formation of social groups were suggested and analysed, see e.g. \cite{Lorenz} and references within. Like our model, these models represent  opinions by numbers, or vectors in Euclidian space \cite{CFL}*{Chapter III}, and association is governed by homophily (i.e., proximity of opinion, see \cite{McPhersan}).

In \cite{social}, however, Alon et al. {investigated models for} the evolution of \emph{centralised} social groups, where admission is governed by a global procedure, such as a voting process or a screening committee. Examples for such groups are ample and diverse, ranging from secret societies, to members of the national academy of sciences and to the inner circle of a martial art association. They include communes and condominiums, justices of the supreme court, associate professors of a math department and nations in the united nations. The evolution of the composition of such groups could vary greatly depending on the nature of the admission process.
One purpose of the authors was to study the different behaviours of models in which the admission is relative, i.e., based on the comparison of several candidates. This they do both for growth models like the one discussed here and for models of population with fixed size. In particular they explore the difference between a voting model, where the candidate who received the highest number is admitted, and a consensus model, where all members must agree on the most suitable candidate. As an intermediate model they suggested the $r$-quantile admission process, where a candidate must obtain a quorum of a certain fraction of the votes in order to be admitted.
Such admission models are common for the selection of justices and governmental committees.

One property of the relative consensus model is its radicalisation behaviour. In a population holding a uniform distribution of opinion in $[-1,1]$, for example, it almost surely converges to atomic measures at $\pm1$. This behaviour is explained by the fact that the only case in which new members are admitted is when both candidates are extremists in the same direction, in which case the milder among the two is admitted. To better understand this radicalisation effect and its interaction with $r$ in the $r$-quantile admission process, the authors put it against a body holding a veto power who allows only admittance of the member whose opinion is minimal. Procedures with similar effects are used in practice by conservative and professional organisations such as the inner-circle of an organization for preserving a traditional craft. In these cases the opinion parameter represents a certain professional standard and the veto is usually exercised by supervising committee. In \cite{social}, it is shown that if the distribution of opinions is uniform, whenever $r>\frac12$, the radicalization effect of the quorum dominates the veto and the process radicalises to an atom at $1$, making the veto an ineffective method for preventing the group's opinions from radicalising upwards.

In this work we provide a rigourous analysis of the $r$-quantile admission process with veto power, which results in a
better understanding of the competition between the drift upwards caused by $r$, the high percentage of support required for admission,
and $\mu$, the natural distribution of the population. We show that the opinion profile of the club always converges weakly and that
while the drift induced by the quorum pressure persists for any distribution, it is never sufficiently strong to make the empirical $r$-quantile opinion of the group tend to infinity. In addition, our tools appear to be useful in the analysis of other quantile driven processes and we hope that they will find applications for other rules of admission and in opinion spaces of higher dimensions.
{In particular we believe them to be relevant for studying the $r$-quantile admission process without veto power. See Problem~\ref{prob: no veto} below.}

\medskip
\textbf{Mathematical interest.} The process we analyse is an example of a \emph{random walk in changing environment} (RWCE), that is, a stochastic process for which in every step the environment which determines the transition probabilities between possible states changes (in our case, the set of possible values of the $r$-quantile is effected by the history and the opinions of the current members of the group).
RWCEs were recently studied by Amir, Benjamini, Gurel-Gurevich and Kozma~\cite{ABGK}, by Avin, Kouck\'y and Lotker \cite{AKL}, by Dembo, Huang and Sidoravicuis \cites{DHS1, DHS2}, and by Redig and V\'ollering \cite{RV} to name but a few. It includes examples which were investigated for their own sake, such as reinforced random walk, bridge-burning random walk, loop-erased random walk, and self-avoiding walk with bond repulsion (see references within \cite{ABGK}).
Typically, one is interested in properties such as convergence and recurrence vs. transience of the process albeit the changing environment. As pointed out by several examples in \cite{ABGK}, these questions are not at all trivial.

While this paper carries out careful analysis of a certain family of RWCE (characterised by $r$ and the distribution of opinions in the population $\mu$), we believe that it may possibly be generalised to a larger family of RWCE that goes beyond quantile driven processes. One neat variant which could be addressed using our methods is a process in which at the ticks of a Poisson clock a point {is added} to the real line according to some absolutely continuous distribution supported on $\R$. A walker makes a simple random walk on these points at the ticks of an independent Poisson clock, moving from each point to one of its immediate neighbours. The probability that the random walk at point $x$ moves to its right neighbour depends on a continuous $f(x):\R\to[0,1]$. Our methods could be used to show that such a walk always converges either a point $x$ satisfying $f(x)=\frac12$, to $\infty$ if $\lim_{x\to\infty} f_x\ge\frac12$ or to $-\infty$ if $\lim_{x\to-\infty} f_x\le\frac12$.

In addition, this work provides an application to a general inequality relating the average of two i.i.d. random variables to their minimum \cite{FF}, demonstrating how it could be applied to control the recurrence of certain random processes.

\subsection{Outline and highlights}\label{sec:tools and methods}
In this section we survey the tools and methods used to obtain our results,
providing a rough outline of key steps in our proof. To make this outline more accessible we restrict ourselves here to the case of continuous $\mu$,
avoiding the difficulties that arise from handling measures with an atomic component.

The first step in proving the convergence of $\nu_t^*$ is its reduction in Section~\ref{sec: reduce} to
the convergence of two probabilities: $\Pro(X+Y\ge 2m)$ and $\Pro(X \ge m)$ for $X,Y$ i.i.d. $\mu$-distributed. To analyse the behaviour of these probabilities we make a time change to the process, examining it at times in which a new member is admitted, thus defining the empirical $q$-quantile process $m_k$.
We then provide four key lemmata (see Section \ref{sec: reduce}), showing that $m_k$ is bounded and classifying regions which it can visit only finitely many times. {
These lemmata are derived by studying the evolution of $m_k$
via a coupling with $y_k$, a certain random walk in changing environment, presented in Section~\ref{sec: tools}.
To define $y_k$ we regard the members of $\cS$ as
the points $\{1,2,\dots, k\}$ translated by some common integer shift $\Delta_k$.
Under this map, we let $y_k =rk+\Delta_k$ be the location of the empirical $r$-quantile after $k$ members have been admitted.
The translation $\Delta_k$ is chosen so that $\E(y_{k+1}-y_{k})=\rho(m_k)$ and the steps of $y_k$ are regular.
Using continuity properties of $\rho$
together with classical properties of drifting random walks}, we show in Section~\ref{sec: lems} that if $\rho(x)\neq 0$ then there exists an interval $I$ containing $x$
which $m_k$ can visit only finitely many times. Moreover, we show that if $\rho(x)>0$ then $m_k$ has positive probability to eventually stay above
$x$, while if $\rho(x)<0$ then $m_k$ has positive probability to eventually stay below $x$.
To conclude our proof we require two more ingredients. The first is an argument, which essentially shows that an unbiased simple random walk
on such an evolving lattice visits every point only finitely many times. This is used to show convergence of $\Pro(X+Y\ge 2m_k)$ and $\Pro(X \ge m_k)$ if $m_k$ ends up in a region
$I$ with $\mu(I)>0$ and $\rho(x)=0$ for all $x\in I$.
Finally to show that $m_k$ is bounded we use the following general comparison inequality established in \cite{FF}*{Theorem 1}.

\begin{thma}\label{thm: FF}
Let $X,Y$ be independent random variables on $\R_+$, which are not compactly supported. Then:
\begin{equation*}
\liminf_{m\to \infty} \frac{\Pro(\min(X,Y) > m)}{\Pro\left(\frac{X+Y}{2}\ge  m\right)} =0.
\end{equation*}
\end{thma}

The quotient here and in \eqref{eq: rho} are the same,
and both represent the probability that the opinion of the next admitted member is above or equal to the empirical $r$-quantile $m$.
Using this we show that there are sufficiently many points with $\rho(x)<0$ on $\R$ so that eventually $m_k$ will stay bounded below one of them.
Finally we use compactness arguments together with all of these observations to show convergence of $m_k$ to a non-empty interval $J$ satisfying
$\mu(J)=0$ and $\rho(x)=0$ for all $x\in J$.

Finally, in Section \ref{sec: exms} we utilise elementary tools to analyse the example mentioned in Proposition~\ref{prop: atom exm} and prove our criterion for uniqueness given in Proposition~\ref{prop: convex exm}.

\subsection{Open Problems}\label{sec:open problems}

We conclude the introduction with two open problems.
\begin{problem}
Does the $r$-quantile admission process with veto power converge for any distribution $\mu$ over $\R^d$ (including $d=1$), where the founder may hold any opinion in $\R^d$? What about other metric spaces?
\end{problem}

\begin{problem}\label{prob: no veto}
Does the $r$-quantile admission process without veto power converge for {all} $\mu,r$?
Does the \emph{drift to the extreme} phenomenon persist in this model?
\end{problem}

{It is also of interest to fully classify the distribution of $\nu_\infty$ in cases for which it is not deterministic, as this would
would complete
the analysis of the $r$-quantile admission process with veto power on $\R$.}

\section{From convergence of measures to convergence of quantiles}\label{sec: reduce}

In this section we reduce Theorems~\ref{thm: main} and~\ref{thm: determinism} to statements concerning with the evolution of the empirical $r$ quantile of the \emph{quantile admission process with veto power}.
All theorems are reduced to a couple of propositions concerning this process, which we further reduce to several technical statements. Tools for proving these statements are provided in Section~\ref{sec: tools} and their proofs are given in Section~\ref{sec: lems}.

\subsection{Notation}
The following notation is used throughout. For $A,B\subseteq \R$, denote
$A+B = \{a+b: \ a\in A, b\in B\}$. We write $a>B$ if $a>b$ for every $b\in B$
and $A>B$ if $a>B$ for every $a\in A$. We follow the convention that the infimum of  an empty set is infinity. We employ the abbreviations \emph{i.o.} for
\emph{infinitely often}, \emph{a.e.} for \emph{almost everywhere} and  \emph{a.s.} for \emph{almost surely}.

Throughout, we fix the probability measure $\mu$ and the quantile $r$, and let
$X$ and $Y$ be two independent $\mu$ distributed random variables.
We define {the events}
\vspace{-2pt}
\begin{align*}
D_m&:={\left\{\frac{X+Y}{2}\ge m\right\}}
&D_{m+}&:={\left\{\frac{X+Y}{2}> m\right\}}\\
F_m&:=\Pro\left(D_m\right)=\mu*\mu([2m,\infty))
&F_{m+}&:=\Pro\left(D_{m+}\right)=\mu*\mu((2m,\infty)).
\end{align*}

We consider our process mostly in times when a member is admitted to the club. To this end,
we define a random sequence $\{x_k\}_{k\in \N_0}$ of the opinions of club members
by order of admission. We further denote the multiset $S_k=\{x_0,\dots, x_{k-1}\}$ and define
the lower $r$-quantile of the set $S_k$ by
\begin{equation}\label{eq: mk def}
m_k=\min \Big\{x\in\R_+\ :\ |\{j < k\ :\ x_j\le x \}|\ge r k\Big\}=\sup\Big\{x\in\R_+\ : \ |\{j< k: \ x_j<x\}|<rk\Big\}.
\end{equation}
Observe that the distribution of $x_k$ is that of $\min(X,Y)$
conditioned on $D_{m_k}=\{\frac{X+Y}{2}\ge m_k\}$.

We denote $\cF_k$ for the sigma algebra generated by $x_0,\dots,x_{k}$
{and write $\nu_k$ for the empirical distribution of opinions in $S_{k+1}=\{x_0,\dots,x_k\}$,
given by $\frac1{|S_{k+1}|}\sum_{s\in S_{k+1}}\delta_s$.}
By the definition of
$\nu_t$, the measures $\nu_t$ and $\nu^*_t$ converge
together and to the same limit. For this reason, the reader
need not be alarmed by the visual similarity between $S_k$ and
$\cS_t$ (used in the introduction) as the latter will play no role in the remainder of the paper.

To simplify the treatment of $\rho$ we introduce
\begin{equation}\label{eq: rho plus}
\rho^+(m)=\rho^+_{\mu,r}(m):= \lim_{q\to m-}\rho_{\mu,r}(q).
\end{equation}
Notice that if $\mu_{\max}<\infty$ then  
\begin{equation}
\rho^+(m)=\frac{\mu\big([m,\infty)\big)^2}{\mu*\mu\big([2m,\infty)\big)}-(1-r).
\end{equation}
We observe the following basic properties of $\rho$ and $\rho+$.

\begin{obs}\label{obs: rho contin}
$\rho$ is lower-semi-continuous and $\rho^+$ is left-continuous.

$\rho(x)\le \rho^+(x)$ for all $x$. Moreover $\rho(x)<\rho^+(x)$ if and only if $\mu(\{x\})>0$.

$\rho(x) \le \lim_{t\to x+}\rho(t)$ for all $x$. Moreover $\rho(x)<\lim_{t\to x+}\rho(t)$ if and only if 
$\mu*\mu(\{2x\})>0$.
\end{obs}

\begin{obs}\label{obs: rho jumps at atom}
If $\mu(\{x\})>0$, then $\rho(x)<\lim_{y\to x+}\rho(y)$.
\end{obs}

\begin{obs}\label{obs: rho monotone for mu null}
$\rho$ and $\rho^+$ are monotone increasing on any interval $I$
satisfying $\mu(I)=0$.
\end{obs}

\begin{obs}\label{obs: lim at -infty}
{$\lim_{q\to -\infty}\rho(q) = \lim_{q\to-\infty}\rho^+(q) = r$.}
\end{obs}

\subsection{Reduction of Theorems \ref{thm: main} and \ref{thm: determinism}}

We being by reducing Theorem~\ref{thm: main} to the following proposition, whose proof we postpone to Section~\ref{subs:pf of main prop}.
 \begin{prop}\label{prop: main}
Almost surely there exists a finite $m\in[0,\mu_{\max}]$ such that the sequence $F_{m_k}$ converges to a limit $F_{\ell}$, for $\ell\in\{m,m+\}.$
Moreover,
\noindent
\begin{align*}
&\text{if $\ell = m<\mu_{\max}$, then $\rho(m)\le 0 \le \rho^+(m)$,}\\
&\text{if $\ell = m+<\mu_{\max}$, then $\rho(m)\le 0 =\lim_{q\to m+}\rho(q)$,}\\[-6pt]
\text{while }&\text{if $m=\mu_{\max}$, then $\rho^+(\mu_{\max})\ge 0$.}
\end{align*}
\end{prop}

\begin{proof}[Proof of Theorem~\ref{thm: main}]
For $s\in[0,\infty)$ denote $B_s= \{(x,y) :\ \min(x,y)>s \}$ and $D_s$ as above.
Observe that for any $s,m'\ge 0$ we have,
\begin{equation}\label{eq: pk def}
p_{s,k}:=\Pro\left(x_k>s\right)=\Pro\left(\min(X,Y)>s \,\middle|\, \frac{X+Y}{2}\ge m_k\right)=
\frac{\Pro\big((X,Y)\in D_{m_k}\cap B_s\big)}{\Pro\left((X,Y)\in D_{m_k}\right)}.
\end{equation}

By Proposition~\ref{prop: main} we obtain that $\lim_{k\to\infty}\Pro\big(D_{m_k}\big)=\lim_{k\to\infty} F_{m_k}$
almost surely exists and is of the form $F_{\ell}$ for
$\ell \in \{m,m+\}$.
Since $\{D_s\}_{s\in\R}$ are ordered by inclusion, we have $\ind(D_{m_k})\xrightarrow[k\to\infty]{\mu}\ind(D_\ell)$. Hence for any measurable set $B\subset \R_+^2$ we have
\begin{equation}\label{eq: limit form}
\ind\big( D_{m_k}\cap B\big)=\ind(D_{m_k})\ind(B)\xrightarrow[k\to\infty]{\mu}\ind (D_\ell)\ind(B)=
\ind\big( D_\ell \cap B\big).
\end{equation}
In particular the numerator of \eqref{eq: pk def} converges as $k$ tends to infinity, almost surely.
To show convergence of right-hand-side of \eqref{eq: pk def}, let us consider separately the case $F_m=0$ and the case
$F_m>0$.

In the case $F_m=0$, we have
$m_k\overset{\text{a.s.}}{\to}\inf\{s\ :\ \Pro\big((X+Y)/2\ge s\big)=0\}=\mu_{\max}$.
By Proposition~\ref{prop: main}, $m_k$ is almost surely bounded and hence
$\mu_{\max}<\infty$. By definition, the distribution of $x_{k}$ is supported on $[2m_k-\mu_{\max},\mu_{\max}]$,
so that $\nu_k$ converges weakly
to $\delta_{\mu_{\max}}$. By the moreover part of Proposition~\ref{prop: main}, this can only occur if condition \ref{item:max criterion} in Theorem~\ref{thm: main} is satisfied.

In the case $F_m>0$ we have, almost surely, for all $s\ge0$,
$$\lim_{k\to\infty}p_{s,k} = \frac{\Pro\big((X,Y)\in D_{\ell}\cap B_s\big)}{\Pro\left((X,Y)\in D_{\ell}\right)}.$$
We conclude that $x_k$ converges weakly, so that, by the law of large numbers $\nu_k$ converges weakly to a limiting distribution $\nu_\infty$. If $\ell=m<\mu_{\max}$, then, by the moreover part of Proposition~\ref{prop: main} the condition $\rho(m)\le 0 \le \rho^+(m)$ is satisfied, so that, considering the limit distribution implied by \eqref{eq: pk def} and \eqref{eq: limit form}, we obtain that item \ref{item:m criterion} of Theorem~\ref{thm: main} holds. Similarly, if $\ell=m+<\mu_{\max}$, then by the moreover part of Proposition~\ref{prop: main} the conditions $\rho(m)\le 0 \le \lim_{q\to m-}\rho(q)$ and $\lim_{q\to m+}\rho(q)=0$ are satisfied, so that, using
\eqref{eq: pk def} and \eqref{eq: limit form},
item \ref{item:m plus criterion}  of Theorem~\ref{thm: main} holds.
Finally, if $m=\mu_{\max}$ then, by similar arguments, item \ref{item:max criterion} of Theorem~\ref{thm: main} holds.
\end{proof}

Next we obtain {the sufficiency of the criterion in Theorem~\ref{thm: determinism}} as a corollary of Theorem~\ref{thm: main}.
\begin{proof}[Proof of the {sufficiency criterion in Theorem~\ref{thm: determinism}}.]
Writing $S_-$ and $S_+$ for the support of the negative and positive parts of $\rho$  respectively, we obtain from the strict monotonicity of $\rho$ that either
$S_+=\emptyset$, or there exists a unique $m'$ such that $\sup S_+=\inf S_-=m'$.

If $S_+$ is empty, then both conditions for items \ref{item:m criterion} and \ref{item:m plus criterion} in Theorem~\ref{thm: main} cannot hold for any $m$, and hence item \ref{item:max criterion} of the theorem must hold, so that $\nu^*_t$ converges to $\delta_{\mu_{\max}}$, as required. Otherwise, $\sup S_+=\inf S_-=m'$, so that, by Theorem~\ref{thm: main}, $\nu^*_t$ converges to $\nu_\infty$ such that
$\nu_\infty((s,\infty))$ is proportional to either $\Pro\left( X>s,\, Y> s ,\, X+Y\ge 2m\right)$
or
$\Pro\left( X>s,\, Y> s,\, X+Y> 2m\right)$.
Since, by definition, $\rho(m)$ is a translation of the quotient of two monotone decreasing functions, where the numerator is right continuous and the denominator is left continuous, it can be monotone decreasing itself only if the denominator is continuous, that is,
$\Pro(X+Y\ge 2m)=\Pro(X+Y>2m)$ for every $m$.
Hence also $\Pro\left( X>s,\, Y> s ,\, X+Y\ge 2m\right) =  \Pro\left( X>s,\, Y> s,\, X+Y> 2m\right)$
so that $\nu_\infty$ is uniquely determined and explicit, as required.
\end{proof}
In order to establish {the necessity of the criterion in Theorem~\ref{thm: determinism}} we use the following proposition, whose proof we postpone to Section~\ref{subs:pf of main non-det}.

\begin{prop}\label{prop: non-determinism}
For any $x\in\R_+$, if $\rho(x)<0$ then $\Pro(m_k\le x\ \text{a.e.})>0$, while if $\rho(x)>0$ then $\Pro(m_k>x\ \text{a.e.})>0$.
Moreover, if $a<b$ are such that $\mu([a,b])>0$ and $\rho(x)=0$ for all $x\in[a,b]$, then $\Pro(m_k\in [a,b] \text{ a.e.})>0$.
\end{prop}

\begin{proof}[Proof of the {necessity of the criterion in Theorem~\ref{thm: determinism}}]
We begin by showing the first part of the theorem. Recall that $\rho_{\mu,1}(0) = 1$ and $\liminf_{q\to\mu_{\max}}\rho_{\mu,1}(q)=0$ by Theorem \ref{thm: FF}.
Thus, under the assumption that $\rho_{\mu,1}(q)$ is not monotone,
there exist $x_{0}<x_{1}$ and $r\in[0,1]$
such that $\rho_{\mu,1}(x_0)<1-r<\rho_{\mu,1}(x_{1})$.
Since $\rho_{\mu,r}(q) = \rho_{\mu,1}(q)-(1-r)$ this means that $\rho_{\mu,r}(x_0)<0<\rho_{\mu,r}(x_1)$.
By Proposition~\ref{prop: non-determinism}, there is a positive probability that $m_k\le x_0$ for almost every $k$, in which case,
by Theorem \ref{thm: main}, $\nu_t$ converges to a measure $\nu_\infty$ of the form given in \ref{item:m criterion} or \ref{item:m plus criterion} for some $m\le x_0$.
 On the other hand, by Proposition \ref{prop: non-determinism} there is a positive probability that $m_k> x_1$
 for almost every $k$, in which case, by Theorem \ref{thm: main}, $\nu_t$ converges to a measure $\nu_\infty$ as in~\ref{item:m criterion} or~\ref{item:m plus criterion} with
$m\ge x_1$ or as in~\ref{item:max criterion}. As a measure cannot satisfy both requirements, we conclude that $\nu_\infty$ is non-deterministic.

Next, we show the moreover part. Suppose that $\rho$ vanishes on $[a,b]$ where $a<b$ and $\mu([a,b])>0$. By Observation \ref{obs: rho contin}, the interval $[a,b]$ contains no atoms of $\mu$. Hence we may restrict ourselves to the case $b<\mu_{\max}$, as otherwise we can apply our arguments to a subinterval of positive measure whose upper end satisfies this property. Observing that $F_s$ is monotone decreasing in $s$ and that $F_s>F_{s'}$ whenever $\mu([s,s'))>0$, we obtain the existence of $a<a'<b'<b$ such that $F_a>F_{a'}>F_{b'}>F_b$.
From Proposition~\ref{prop: main} we obtain that $F_{m_k}$ converges to a limit $F_\ell$.
By the second part of Proposition \ref{prop: non-determinism} both the event
$A=\{m_k\in [a,a']\text{ a.e.}\}$ and the event
$B=\{m_k\in [b',b]\text{ a.e.}\}$ have positive probability.
In the former case we obtain $F_\ell\ge F_{a'}$ while in the latter $F_\ell\le F_{b'}$. We conclude that the values of $F_\ell$ under the event $A$ and under the event $B$ must be almost surely distinct. Observing that in both cases $\nu_\infty([b,\infty))=\mu([b,\mu_{\max}))^2/F_\ell$, where the numerator is a non-zero constant, independent of $F_\ell$, we conclude that
$\nu_\infty([b,\infty))$ is not a constant random variable, as required.
\end{proof}

Thus we are left with proving Propositions~\ref{prop: main} and~\ref{prop: non-determinism}. In the next section we reduce these propositions to several technical statements.
%


\subsection{Properties of the empirical quantile process}
In this section we present several lemmata concerning with the empirical quantile process $m_k$ which will be of use in the proof of Propositions~\ref{prop: main} and~\ref{prop: non-determinism}.
We begin by introducing several definitions.

\begin{defin}[Barrier]\label{def: bar}
A point $x\in\R$ is called a \emph{barrier} if there exists $k_0$ such that
either $m_k \le x$ for all $k\ge k_0$, or $m_k> x$ for all $k\ge k_0$.

 A point $x\in\R$ is called a \emph{right-barrier} if it is a.s. a barrier and
 $\Pro(\text{$m_k \le x$ a.e.})>0$. Similarly, $x\in\R$ is called a \emph{left-barrier} if it is a.s. a barrier and $\Pro(\text{$m_k> x$ a.e.})>0$.
\end{defin}

\begin{defin}[Separator]\label{def: sep}

An interval $I$ is called a \emph{separator}
if there exists $k_0$ such that
either $m_k < I$ for all $k\ge k_0$, or $m_k>I$ for all $k\ge k_0$.

\noindent
 An interval $I$ is called a \emph{right-separator} if it is a.s. a separator and
$\Pro({\exists k_0:\:} \text{$m_k< I$ for all $k\ge k_0$})\!>\!0$. Similarly, $I$ is called a \emph{left-separator} if
 it is a.s. a separator and $\Pro({\exists k_0:\:} \text{$m_k> I$ for all $k\ge k_0$})>0$.
\end{defin}


To prove our propositions, we require the following lemmata.

\begin{lem}[Boundedness]\label{lem:as bounded}
$m_k$ is almost surely bounded.
\end{lem}

\begin{lem}[Negative drift]\label{lem: right separators}
Let {$x_0 \le \mu_{\max}$}.
\begin{enumerate}[nolistsep, label={\emph(\alph*\emph)},ref=\alph*]
\item\label{lem: right separators: item plus} If $\rho^+(x_0) < 0$,
then there exists $a< x_0$ such that $(a,x_0]$ is a right-separator.
\item\label{lem: right separators: item barrier} If $\rho(x_0) < 0$,
then $x_0$ is a right-barrier.
\item\label{lem: right separators: item right limit} If $\lim_{x\to x_0+}\rho(x) < 0$,
then there exists $b> x_0$ such that $(x_0,b)$ is almost surely a separator.
\end{enumerate}
\end{lem}
\begin{lem}[Positive drift]\label{lem: left separators}
Let {$x_0 < \mu_{\max}$}.
\begin{enumerate}[nolistsep, label={\emph(\alph*\emph)},ref=\alph*]
\item\label{lem: left separators: item plus} If $\rho(x_0) > 0$,
then {there exist $a<x_0<b$ such that 
$(a,b)$ is a left-separator}.
\item\label{lem: left separators: item barrier} If $\rho^+(x_0) > 0$,
then $x_0$ is almost surely a barrier.
\item\label{lem: left separators: item right limit} If $\lim_{x\to x_0+}\rho(x) > 0$,
then there exists $b> x_0$ such that $(x_0,b)$ is almost surely a separator.
\end{enumerate}
\end{lem}

\begin{lem}[Zero drift]\label{lem: zero drift barrier}
Let $[a,b]$ be an interval such that $\mu((a,b))>0$, and $\rho(x)=\rho^+(x)=0$ for all $x\in[a,b]$.
Then $[a,b]$ almost surely contains a barrier. Moreover, $\Pro(m_k\in [a,b] \text{ a.e.})>0$.
\end{lem}

From these we also derive the following corollary.

\begin{cor}\label{cor: measure is barrier}
let $a\le b$ such that $F_{a}<F_{b+}$, then $[a,b]$ a.s. contains a barrier.
\end{cor}
\begin{proof}
We consider three cases. If there exists $x\in[a,b]$ such that $\rho(x)<0$, then, by Lemma~\ref{lem: right separators}\ref{lem: right separators: item barrier}, the point $x$ is a.s. a barrier. Similarly, if
there exists $x\in[a,b]$ such that $\rho^+(x)>0$, then,
by Lemma~\ref{lem: left separators}\ref{lem: left separators: item barrier}, the point $x$ is a.s. a barrier. Since $\rho(x)\le \rho^+(x)$ for all $x$ (by Observation~\ref{obs: rho contin}), we obtain in the remaining case, $\rho(x)=\rho^+(x)=0$ for all $x\in[a,b]$.
We observe that, by definition, $F_x$ is monotone decreasing on any interval of $\mu$-measure zero. Hence by the assumption $F_{a}<F_{b+}$ together with \eqref{eq: rho plus}, we have $\mu[a,b]>0$. By Observation~\ref{obs: rho contin} we have $\mu(\{a\})=\mu(\{b\})=0$ so that $\mu(a,b)>0$. Thus, by Lemma~\ref{lem: zero drift barrier}, the interval $(a,b)$ almost surely contains a barrier.
\end{proof}
\subsection{Proof of Proposition~\ref{prop: main}}\label{subs:pf of main prop}
By Lemma~\ref{lem:as bounded}, the sequence $m_k$ is almost surely bounded.
In addition, $F_\ell$ is left-continuous and monotone decreasing, so that if $F_{m_k}$ converges, then its
limit must be of the form $F_\ell$ for $\ell\in \{m,m+\}$ for a finite $m\in\R$.

Let $\{m_{k_i}\}_{i\in\N}$ and $\{m_{k'_i}\}_{i\in\N}$ be two convergent subsequences of $\{m_k\}_{k\in\N}$
whose limits we denote by $m$ and $m'$ respectively, and assume without loss of generality that $m'\le m$.
To show convergence of $F_{m_k}$ it would suffice to show that
\begin{equation}\label{eq: F goal}
\lim_{i\to\infty} F_{m_{k'_i}}=\lim_{i\to\infty} F_{m_{k_i}}.
\end{equation}

For any $\ep>0$, we have $m_k<m'+\ep$ for
infinitely many values of $k$ and $m_k>m-\ep$ for
infinitely many values of $k$ so that the interval $(m',m)$
does not contains any barrier.
By Corollary~\ref{cor: measure is barrier}, this implies that $F_{m'}=F_m$, almost surly.

To derive convergence we consider two cases. If $F_{m}= F_{m+}$,
then
\vspace{-5pt}
$$F_{m}=F_{m+}\le F_{m'+} \le F_{m'}=F_{m},$$
\vspace{-20pt}

\noindent so that
{
\eqref{eq: F goal} holds with both sides equal to $F_m=F_{m+}=F_{m'}=F_{m'+}$}. On the other hand, if $F_{m}\neq F_{m+}$,
then $m$ is a barrier by Corollary~\ref{cor: measure is barrier}.
Hence either $m_k\le m$ for almost all $k$, or $m_k>m$ for almost all $k$
and $m=m'$. In the latter case we immediately obtain \eqref{eq: F goal} with
both sides equal $F_{m+}$.
In the former, since $F_{m}\le F_{m'+}\le F_{m'}=F_m$, we obtain that
\eqref{eq: F goal} holds with both sides equal to $F_{m}$.

\smallskip
Next, we establish the moreover part of the proposition.
Let $\{m_{k_i}\}_{i\in\N}$ be a {monotone} convergent subsequence of $\{m_k\}_{k\in\N}$ and denote its limit by $\ell\in \{m,m+\}$,
and observe that, by definition, there cannot be a separator $(a,b)$ containing $m$. 

First, observe that if $\mu_{\max}<\infty$ and $\rho^+(\mu_{\max})<0$ then by 
Lemma~\ref{lem: right separators}\ref{lem: right separators: item plus} there exists $a<\mu_{\max}$ such that $(a,\mu_{\max}]$ is a separator. Therefore if $m=\mu_{\max}<\infty$ then $\rho^+(\mu_{\max})\ge 0$. 

Henceforth we assume that $\rho$ is defined at $m$.
Denote $S^+=\{x\ :\ \rho(x)>0\}$. For each $x\in S^+$ write $I_x^+:=(a_x,b_x)$ for the separator guaranteed by Lemma~\ref{lem: left separators}\ref
{lem: left separators: item plus} and observe that $m\notin (a_x,b_x)$ almost surely. The collection $\{I_x^+: \ x\in S^+\}$ is an open cover of $S^+$. Since $\R$ is second-countable,
this collection has a countable subcover. Since $m$ is almost surely not in any of the sets of this subcover, we deduce that $m\notin S^+$ almost surely, that is, 
\begin{equation}\label{eq: first case cover}
\rho(m)\le 0.
\end{equation}

Similarly, write $S^-=\{x\ :\ \rho^+(x)<0, \lim_{t\to m+} \rho(t) \neq 0\}$. 
For each $x\in S^-$ write $I_x^-:=(a_x,b_x)$ where $(a_x,m]$ is the separator guaranteed by Lemma~\ref{lem: right separators}\ref{lem: right separators: item plus},
and $(m,b_x)$ is the separator guaranteed by either Lemma~\ref{lem: right separators}\ref{lem: right separators: item right limit}
or Lemma~\ref{lem: left separators}\ref{lem: left separators: item right limit}, depending on the sign of $\lim_{t\to m+} \rho(t)$. Observe that
$m\notin (a_x,b_x)$ almost surely. As before, the collection of separtors $\{I_x^-: \ x\in S^-\}$ is an open cover of $S^-$
so that $m\notin S^-$ almost surely. We deduce that 
\begin{equation}\label{eq: second case cover}
\text{either
$\rho^+(m)\ge 0$ or $\lim_{t\to m+} \rho(t) = 0$.}
\end{equation}

Next, denote by $A$ the set $x$ such that $2x$ is an atom of $\mu*\mu$. Observe that for all $x\in A$ we have $F_x \neq F_{x+}$. For any $x\in A$ which satisfies $\rho^+(x)< 0$, we have, by Lemma~\ref{lem: right separators}\ref{lem: right separators: item plus},
that there exists $a<m$ such that $(a,x]$ is a separator so that $m_{k_i}$ can only converge to $x$ from above almost surely, i.e., if $m=x$ then
$F_\ell = F_{m+} \neq F_{m}$ almost surely.
Similarly, for any $x\in A$ which satisfies $\lim_{t\to x+} \rho(t) \neq 0$ we have, by Lemma~\ref{lem: right separators}\ref{lem: right separators: item right limit}
and Lemma~\ref{lem: left separators}\ref{lem: left separators: item right limit}, that almost surely
there exists $b>x$ such that $(x,b)$ is a separator. Thus, almost surely, $m_{k_i}$ can only converge to $x$ from below, 
so that if $m=x$ then
$F_\ell = F_{m} \neq F_{m+}$ almost surely. 
Since $A$ is countable we deduce that in case that $m\in A$, the proposition is satisfied.

Finally for any $x\notin A\cup S^+ \cup S^-$, by Observation~\ref{obs: rho contin} and the fact that $A$ contains the atoms of $\mu$, we have
$\rho(x)=\rho^+(x)=\lim_{t\to x+} \rho(t)$. By combining \eqref{eq: first case cover}
and \eqref{eq: second case cover} we deduce that in case $m\notin  A\cup S^+ \cup S^-$ we have
$\rho(m)=\rho^+(m)=\lim_{t\to m+}\rho(t)=0$. Since in this case
both $F_m=F_{m+}$, the proposition is satisfied.
\qed.

\subsection{Proof of Proposition~\ref{prop: non-determinism}}\label{subs:pf of main non-det}
The case $\rho(x)<0$ is immediate from Lemma \ref{lem: right separators}\ref{lem: right separators: item barrier}, the case $\rho(x)>0$ of the first part is immediate from Lemma \ref{lem: left separators}\ref{lem: left separators: item plus}.
The moreover part is immediate from Lemma \ref{lem: zero drift barrier}. \qed.

\section{Preliminaries}\label{sec: tools}
In this section we establish the probabilistic infrastructure required to prove lemmata \ref{lem:as bounded}-\ref{lem: zero drift barrier}.
In Section \ref{sec: gen walks} we introduce facts about general random walks with drift.
In Section~\ref{subs:quantile random walks} we construct a coupling of the quantile process with a random walk in changing environment and introduce relevant notation for its analysis.
In Section \ref{sec: meas} we present a simple claim about continuous probability measures.

\subsection{On general random walks}\label{sec: gen walks}

Let $\{A_k\}_{k\ge k_0}$ be a sequence of events, {adapted to a filtration $\{\cF_k\}$}. Denote the stopping time $T_{\{A_k\}}:=\min\{k \ge k_0: \ A_k \text{ occurred} \}$.
The first result we recall is Azuma's inequality \cite{Azuma} concerning martingales.

\begin{lem}[Azuma]\label{lem: A}
Let $X_k$ be a random process started at $X_0=0$. Assume that
\begin{mitemize}
\item ($X_k$ is a martingale) for all $k$ we have $\E(X_{k+1} \ |\ X_0,\dots, X_k )= X_k $.
\item (uniformly bounded steps) almost surely $|X_{k+1}-X_k| \le 1$ for all $k\in\N$.
\end{mitemize}
Then for every $\al>0$ we have $\Pro( T_{\{X_k\ge \al\}} < n ) \le \exp(-\frac{\al^2}{2n})$.
\end{lem}

We also recall Hoeffding's inequality for i.i.d. binomial random variables.
\begin{lem}[Hoeffding]\label{lem: H}
If $X\sim \text{Bin}(n,p)$, then for any $\ep>0$ we have
{$\Pro( X-p n  \ge \ep n) \le e^{-2\ep^2 n}$
and $\Pro(X-pn \le -\ep n) \le e^{-2\ep^2 n}$.}
\end{lem}

We need the following variation on the classical ``gambler's ruin'' problem (a.k.a. Cram\'er-Lundberg inequality). Its proof, being standard, is omitted (see~\cite{ruin} and~\cite{Fel}*{Ch. XIV.2} for similar results).
\begin{lem}\label{lem: ruin}
Let $X_k$ be a random process started at $X_0=0$ with the following properties:
\begin{mitemize}
\item there exists $\eta>0$ such that for all $k\in\N$ we have:
$ \E(X_{k+1} \ |\ X_0,\dots, X_k )\le X_k -\eta$.
\item  there exists $M>0$ such that
almost surely $|X_{k+1}-X_k| \le M$ for all $k\in\N$.
\end{mitemize}
Then there exists $q=q_{M,\eta}\in (0,1)$ such that for all $\ell>0$ we have: \[ \Pro(T_{\{X_k\ge \ell\}} < \infty \ \big| \ X_{0}=0)< q^\ell.\]
\end{lem}

We use this lemma to obtain the following bound.

\begin{lem}\label{lem: ruin2}

Let $\ell_0\in\N$ and let $\{\ell_i\}_{i\in\N}$ be a sequence of random variables taking values in $\N_0$, adapted to a filtration $\{\cF^{(i)}_0\}_{i\in\N}$
(i.e., $\ell^{(1)},\ell^{(i)}$ are $\cF_0^{(i)}$-measurable).
Further let $\{X_k^{(i)}\}_{i\in\N, k\in\N_0}$ be a collection of random processes on $\R$, started at $X_0^{(i)}=0$, which are
adapted to a filtration $\{\cF^{(i)}_1\}_{i\in\N}$ (i.e., $X^{(1)},\dots,X^{(i)}$ are $\cF_1^{(i)}$-measurable), and
assume that for all $i,k\in\N$ $|X^{(i)}_{k}-X^{(i)}_{k-1}| \le  1$ almost surely, and assume that $\cF^{(i)}_0 \subseteq \cF^{(i)}_1 \subseteq \cF^{(i+1)}_0$ for all $i\in\N$.
Finally, let $\tau_i$ be a stopping time for $\{X_k^{i}\}_{k\in\N}$.

Assume that the following conditions hold for some $\eta,p>0$:
\begin{enumerate}
\item For all $k<\tau_i$ we almost surely have $\E\left(X^{(i)}_{k+1} \ \big| \ X^{(i)}_0,\dots, X^{(i)}_k, \cF_0^{(i)} \right)\le X^{(i)}_k -\eta$,
\item We almost surely have $\Pro\left(\ell_i>0\ |\ \cF^{(i-1)}_1\right)>p$.
\end{enumerate}

Then, there exists $q(\eta,p)\in (0,1)$ such that
\vspace{-5pt}
\[
\Pro\Big(\exists i\in\N, k< \tau_i: \ X^{(i)}_k\ge \sum_{j=0}^i \ell_j\Big)\le q^{\ell_0}.
\]
\end{lem}
\begin{proof}
Denote $E_{j,L}=\{\forall k<\tau_j\ :\ X^{(j)}_k<L\}$ and
$F_{j,L}=\{\ell_j = L\}$.
For {$\{L_i\} = \{L_i\}_{i\in\N}$}, a sequence of non-negative integers,
denote {$\overline L_j=\sum_{i=0}^j L_i$, $\overline E_{n,\{L_i\}}=\bigcap_{i=0}^n E_{i,\overline L_i}$ and
$\ov F_{n,\{L_i\}}=\bigcap_{i=0}^n F_{i,L_i}$.}
We compute:
\begin{equation}\label{eq: big prod}
\Pro\Big(\forall i\in\N, k< \tau_i: \ X^{(i)}_k< \sum_{j=0}^i \ell_j\Big)=
\sum_{\{L_i\}_{i\in\N}}\prod_{j=0}^\infty \Pro(F_{j,L_j}\ |\ \ov F_{j-1,\{L_i\}}, \ov E_{j-1,\{L_i\}})
\Pro( E_{j,L_j}\ |\ \ov F_{j,\{L_i\}}, \ov E_{j-1,\{L_i\}})
\end{equation}

We now bound from below each term in this product.
Observe that, for each $j$, the process $(X^{(j)} \mid \cF^{(j)}_0)$ satisfies the conditions of Lemma~\ref{lem: ruin} almost surely, so that there
exists $w=w(\eta)\in(0,1)$ such that
\begin{equation}\label{eq: Es}
\Pro\big(E_{j,L_j} \mid \ov F_{j,\{L_i\}}, \ov E_{j-1,\{L_i\}}\big) \ge \,1-w^{L}
\end{equation}
 for each $j,L\in\N$.
Let $\{B_i\}_{i\in\N}$ be a sequence of i.i.d. Bernoulli$(p)$ random variables independent from everything else and write $\overline B_j=\sum_{i=1}^j B_j$.
By our assumption,
$\Pro(\ell_i>0\ |\ \cF^{(i-1)}_1 )\ge \Pro(B_j=1 )$ almost surely.
Thus, for any given sequence of numbers $\{L_i\}_{i\in\N}$ we have:
\begin{equation}\label{eq: Fs}
\Pro\big(F_{j,L_j}\ \big|\ \ov F_{j-1,\{L_i\}}, \ov E_{j-1,\{L_i\}}\big) \ge  \Pro\big(B_j=\ind(L_i>0)\big).
\end{equation}

Using \eqref{eq: Es} and {the stochastic domination provided by} \eqref{eq: Fs}, we obtain a bound on the RHS of \eqref{eq: big prod} {in terms of }the random sequence $\{B_i\}$:
\begin{align*}
1-\Pro\Big(\exists i\in\N, k< \tau_i: \ X^{(i)}_k\ge \sum_{j=0}^i \ell_j\Big)
&\ge \sum_{\{b_i\} \in \{0,1\}^\N }\prod_{j=0}^\infty \Pro(B_j=b_i)
(1-w^{\overline B_j})
\\ & \ge \sum_{i= \ell_0}^\infty \sum_{n=1}^\infty p (1-p)^{n-1} (1-w^i)^n
\\ & = p \sum_{i= \ell_0}^\infty \sum_{n=0}^\infty (1-p)^n(1-w^i)^n
\\ &= \frac{p}{p+w^\ell_0-pw^{\ell_0}} = 1-  \frac{w^{\ell_0}-pw^{\ell_0}}{p+w^{\ell_0}-pw^{\ell_0}}
\\ & \ge 1-(1-p)\min\left(\frac{w^{\ell_0}}{p},1 \right).
\end{align*}
Finally we let $s$ be such that $w^{s}=p^2$ and observe that
$$(1-p)\min\left(\frac{w^{\ell_0}}{p},1 \right)\le \max\left(\sqrt{w},\sqrt[s]{1-p}\right)^{\ell_0}.$$
Setting $q=\max\left(\sqrt{w},\sqrt[s]{1-p}\right)$, the proposition follows.
\end{proof}

\subsection{The quantile random walks}\label{subs:quantile random walks}
We have seen that the empirical distribution of opinions depends entirely on the quantile random process $m_k$.
However, this process defies direct analysis as it has irregular steps and high dependence on the past.
To tame it, we introduce a pair of transformations of $m_k$, which we refer to as \emph{the quantile random walks}.
These are defined as follows:
\begin{equation} \label{eq: y def}
\begin{split}
&y_k = rk - |\{j< k\ :\ x_j \le m_j \}| , \\ &y^+_k=rk -  |\{j< k\ :\ x_j < m_j \}|.
\end{split}
\end{equation}
Observe that
\begin{equation}\label{eq: y pm}
\begin{cases}
& \Pro(y_{k+1}=y_{k}+r)=\Pro(\min(X,Y)>m_k \mid \frac{X+Y}{2}\ge m_k),\\
& \Pro(y_{k+1}=y_{k}+r-1)=\Pro(\min(X,Y)\le m_k\mid \frac{X+Y}{2}\ge m_k),\\
& \Pro(y^+_{k+1}=y^+_{k}+r)=\Pro(\min(X,Y)\ge m_k \mid \frac{X+Y}{2}\ge m_k),\\
& \Pro(y^+_{k+1}=y^+_{k}+r-1)=\Pro(\min(X,Y) < m_k\mid \frac{X+Y}{2}\ge m_k).\\
\end{cases}
\end{equation}
The drifts of the walks $y,y^+$ are
\begin{equation} \label{eq: rho def}
\begin{split}
\rho(x)&=\E[(y_{k+1}-y_k) \mid m_k=x, \cF_{k-1}],\\
\rho^+(x)&=\E[(y^+_{k+1}-y^+_k) \mid m_k=x, \cF_{k-1}],
\end{split}
\end{equation}
where $\rho$ and $\rho^+$ are those given in \eqref{eq: rho} and \eqref{eq: rho plus}, respectively.
We remark that the need for introducing the $+$ variant of each notation arises from the fact that our theorems are stated in full generality, allowing $\mu$
to have an atomic part. In case that $\mu$ is a continuous measure -- both notions almost surely coincide.

We further denote
\begin{align*}
& \psi_k(x)= |\{j< k\ :\ x_j< x \}| - |\{j< k\ :\ x_j < m_j \}|, \\
& \psi^+_k(x)= |\{j< k\ :\ x_j\le x \}| - |\{j< k\ :\ x_j \le  m_j \}|.
\end{align*}
Observe that, by the definition of the empirical quantile $m_k$ in \eqref{eq: mk def},
\begin{equation}\label{eq: m and psi}
m_k\le x \iff y_k\le \psi^+_k(x), \quad \text{and} \quad m_k\ge  x \iff y^+_k> \psi_k(x).
\end{equation}
Also notice that
\begin{equation}\label{eq: diff psi}
\begin{cases}
&\psi_{k+1}(x) = \psi_{k}(x) + \ind(x_{k}< x) - \ind(x_{k}< m_{k}),\\
&\psi^+_{k+1}(x) = \psi^+_{k}(x) + \ind(x_{k}\le x) - \ind(x_{k}\le m_{k}),
\end{cases}
\end{equation}
so that
\begin{equation}\label{eq: psi moves}
m_k\ge x\ \Rightarrow \ \psi_{k+1}(x) \le \psi_{k}(x), \qquad
m_k\le x\ \Rightarrow \ \psi^+_{k+1}(x) \ge \psi^+_{k}(x),
\end{equation}
and {for any} $a<b$ both $\big(\psi^+_{k}(b)-\psi^+_{k}(a)\big)$
and $\big(\psi_{k}(b)-\psi_{k}(a)\big)$ are monotone increasing.
{The evolution of $m_k$, $\psi_k$ and $y_k$ for a particular sequence of members $\{x_k\}_{k=0}^5$}
is illustrated in figure~\ref{fig: psi}.
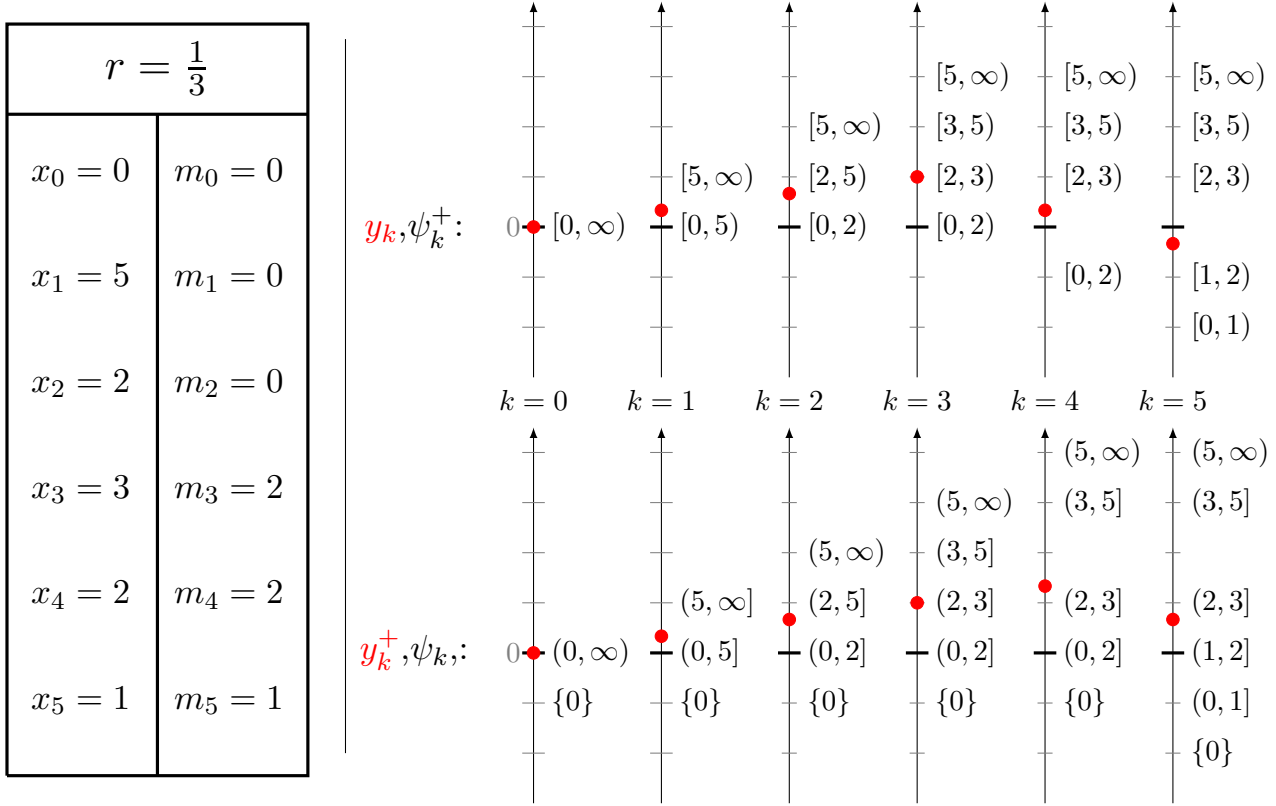
\begin{figure}
\usetikzlibrary{arrows,shapes,positioning,calc}
\begin{tikzpicture}
\draw[very thick] (0,-7.3) -- (4,-7.3) -- (4,2.7) -- (0,2.7) -- (0,-7.3);
\draw[very thick] (0,1.5) -- (4,1.5);
\draw[very thick] (2,1.5) -- (2,-7.3);
\coordinate (T) at (2,2);
\node[below=-17pt of T]{\scalebox{1.5}{$r=\frac13$}};
\node[below=27pt of T]{\scalebox{1.2}{$x_0=0$\hspace{14pt}$m_0=0$}};
\node[below=67pt of T]{\scalebox{1.2}{$x_1=5$\hspace{14pt}$m_1=0$}};
\node[below=107pt of T]{\scalebox{1.2}{$x_2=2$\hspace{14pt}$m_2=0$}};
\node[below=147pt of T]{\scalebox{1.2}{$x_3=3$\hspace{14pt}$m_3=2$}};
\node[below=187pt of T]{\scalebox{1.2}{$x_4=2$\hspace{14pt}$m_4=2$}};
\node[below=227pt of T]{\scalebox{1.2}{$x_5=1$\hspace{14pt}$m_5=1$}};

\draw (4.5,-7) -- (4.5,2.5) ;
\node[ align=center, xshift=154, yshift=0, rotate=0, minimum height=4em](potok1){\scalebox{1.2}{$\color{red}y_k$,$\psi^+_k$:}};
\node[ align=center, xshift=154, yshift=-160, rotate=0, minimum height=4em](potok1){\scalebox{1.2}{$\color{red}y_k^+$,$\psi_k$,:}};

\coordinate (L1) at (7,0);
\draw[latex-] (7,3) -- (7,-2)
node[below=1pt]{$k=0$};
\foreach \x in  {-2,-1,0,1,2,3,4}
\draw[shift={(7,2*\x/3)},color=gray] (-0.15,0pt) -- (0.15,0pt);
\draw[very thick] (6.85,0) -- (7.15,0);
\node[left=1pt of L1, color=gray]{0};
\node[right=3pt of L1]{$[0,\infty)$};
\draw[red,fill=red](7,0) circle (.5ex);

\draw[latex-] (8.7,3) -- (8.7,-2)
node[below=1pt]{$k=1$};
\foreach \x in  {-2,-1,0,1,2,3,4}
\draw[shift={(8.7,2*\x/3)},color=gray] (-0.15,0pt) -- (0.15,0pt);
\draw[very thick] (8.55,0) -- (8.85,0);
\coordinate (L2) at (8.7,0);
\node[right=3pt of L2]{$[0,5)$};
\node[above right=1/3 and 3pt of L2]{$[5,\infty)$};
\draw[red,fill=red] (L2)++(0,2/9) circle (.5ex);

\draw[latex-] (10.4,3) -- (10.4,-2)
node[below=1pt]{$k=2$};
\foreach \x in  {-2,-1,0,1,2,3,4}
\draw[shift={(10.4,2*\x/3)},color=gray] (-0.1,0pt) -- (0.1,0pt);
\coordinate (L3) at (10.4,0);
\draw[very thick] (10.25,0) -- (10.55,0);
\node[right=3pt of L3]{$[0,2)$};
\node[above right=1/3 and 3pt of L3]{$[2,5)$};
\node[above right=1 and 3pt of L3]{$[5,\infty)$};
\draw[red,fill=red] (L3)++(0,4/9) circle (.5ex);

\draw[latex-] (12.1,3) -- (12.1,-2)
node[below=1pt]{$k=3$};
\foreach \x in  {-2,-1,0,1,2,3,4}
\draw[shift={(12.1,2*\x/3)},color=gray] (-0.1,0pt) -- (0.1,0pt);
\draw[very thick] (11.95,0) -- (12.25,0);
\coordinate (L4) at (12.1,0);
\node[right=3pt of L4]{$[0,2)$};
\node[above right=1/3 and 3pt of L4]{$[2,3)$};
\node[above right=1 and 3pt of L4]{$[3,5)$};
\node[above right=5/3 and 3pt of L4]{$[5,\infty)$};
\draw[red,fill=red] (L4)++(0,6/9) circle (.5ex);

\draw[latex-] (13.8,3) -- (13.8,-2)
node[below=1pt]{$k=4$};
\foreach \x in  {-2,-1,0,1,2,3,4}
\draw[shift={(13.8,2*\x/3)},color=gray] (-0.1,0pt) -- (0.1,0pt);
\draw[very thick] (13.65,0) -- (13.95,0);
\coordinate (L5) at (13.8,0);
\node[below right=1/3 and 3pt of L5]{$[0,2)$};
\node[above right=1/3 and 3pt of L5]{$[2,3)$};
\node[above right=3/3 and 3pt of L5]{$[3,5)$};
\node[above right=5/3 and 3pt of L5]{$[5,\infty)$};
\draw[red,fill=red] (L5)++(0,2/9) circle (.5ex);

\draw[latex-] (15.5,3) -- (15.5,-2)
node[below=1pt]{$k=5$};
\foreach \x in  {-2,-1,0,1,2,3,4}
\draw[shift={(15.5,2*\x/3)},color=gray] (-0.1,0pt) -- (0.1,0pt);
\draw[very thick] (15.35,0) -- (15.65,0);
\coordinate (L6) at (15.5,0);
\node[below right=3/3 and 3pt of L6]{$[0,1)$};
\node[below right=1/3 and 3pt of L6]{$[1,2)$};
\node[above right=1/3 and 3pt of L6]{$[2,3)$};
\node[above right=3/3 and 3pt of L6]{$[3,5)$};
\node[above right=5/3 and 3pt of L6]{$[5,\infty)$};
\draw[red,fill=red] (L6)++(0,-2/9) circle (.5ex);


\coordinate (M1) at (7,-17/3);
\draw[latex-] (7,-8/3) -- (7,-23/3);
\foreach \x in  {-2,-1,0,1,2,3,4}
\draw[shift={(7,-17/3+\x*2/3)},color=gray] (-0.15,0pt) -- (0.15,0pt);
\draw[very thick] (6.85,-17/3) -- (7.15,-17/3);
\node[left=1pt of M1, color=gray]{0};
\node[below right=1/3 and 3pt of M1]{$\{0\}$};
\node[right=3pt of M1]{$(0,\infty)$};
\draw[red,fill=red] (M1)++(0,0) circle (.5ex);

\draw[latex-] (8.7,-8/3) -- (8.7,-23/3);
\foreach \x in  {-2,-1,0,1,2,3,4}
\draw[shift={(8.7,-17/3+2*\x/3)},color=gray] (-0.15,0pt) -- (0.15,0pt);
\draw[very thick] (8.55,-17/3) -- (8.85,-17/3);
\coordinate (M2) at (8.7,-17/3);
\node[below right=1/3 and 3pt of M2]{$\{0\}$};
\node[right=3pt of M2]{$(0,5]$};
\node[above right=1/3 and 3pt of M2]{$(5,\infty]$};
\draw[red,fill=red] (M2)++(0,2/9) circle (.5ex);

\draw[latex-] (10.4,-8/3) -- (10.4,-23/3);
\foreach \x in  {-2,-1,0,1,2,3,4}
\draw[shift={(10.4,-17/3+2*\x/3)},color=gray] (-0.1,0pt) -- (0.1,0pt);
\coordinate (M3) at (10.4,-17/3);
\draw[very thick] (10.25,-17/3) -- (10.55,-17/3);
\node[below right=1/3 and 3pt of M3]{$\{0\}$};
\node[right=3pt of M3]{$(0,2]$};
\node[above right=1/3 and 3pt of M3]{$(2,5]$};
\node[above right=1 and 3pt of M3]{$(5,\infty)$};
\draw[red,fill=red] (M3)++(0,4/9) circle (.5ex);

\draw[latex-] (12.1,-8/3) -- (12.1,-23/3);
\foreach \x in  {-2,-1,0,1,2,3,4}
\draw[shift={(12.1,-17/3+2*\x/3)},color=gray] (-0.1,0pt) -- (0.1,0pt);
\draw[very thick] (11.95,-17/3) -- (12.25,-17/3);
\coordinate (M4) at (12.1,-17/3);
\node[below right=1/3 and 3pt of M4]{$\{0\}$};
\node[right=3pt of M4]{$(0,2]$};
\node[above right=1/3 and 3pt of M4]{$(2,3]$};
\node[above right=1 and 3pt of M4]{$(3,5]$};
\node[above right=5/3 and 3pt of M4]{$(5,\infty)$};
\draw[red,fill=red] (M4)++(0,6/9) circle (.5ex);

\draw[latex-] (13.8,-8/3) -- (13.8,-23/3);
\foreach \x in  {-2,-1,0,1,2,3,4}
\draw[shift={(13.8,-17/3+2*\x/3)},color=gray] (-0.1,0pt) -- (0.1,0pt);
\draw[very thick] (13.65,-17/3) -- (13.95,-17/3);
\coordinate (M5) at (13.8,-17/3);
\node[below right=1/3 and 3pt of M5]{$\{0\}$};
\node[right=3pt of M5]{$(0,2]$};
\node[above right=1/3 and 3pt of M5]{$(2,3]$};
\node[above right=5/3 and 3pt of M5]{$(3,5]$};
\node[above right=7/3 and 3pt of M5]{$(5,\infty)$};
\draw[red,fill=red] (M5)++(0,8/9) circle (.5ex);

\draw[latex-] (15.5,-8/3) -- (15.5,-23/3);
\foreach \x in  {-2,-1,0,1,2,3,4}
\draw[shift={(15.5,-17/3+2*\x/3)},color=gray] (-0.1,0pt) -- (0.1,0pt);
\draw[very thick] (15.35,-17/3) -- (15.65,-17/3);
\coordinate (M6) at (15.5,-17/3);
\node[below right=3/3 and 3pt of M6]{$\{0\}$};
\node[below right=1/3 and 3pt of M6]{$(0,1]$};
\node[right=3pt of M6]{$(1,2]$};
\node[above right=1/3 and 3pt of M6]{$(2,3]$};
\node[above right=5/3 and 3pt of M6]{$(3,5]$};
\node[above right=7/3 and 3pt of M6]{$(5,\infty)$};
\draw[red,fill=red] (M6)++(0,4/9) circle (.5ex);
\end{tikzpicture}
\captionsetup{width=.8\linewidth}
\caption{The evolution of $\psi_k, \psi_k^+, \color{red}{y_k}$ and $\color{red}{y_k^+}$ over the first five steps of a sample of the process $x_k$. In this particular example $r=\frac 1 3$, and the first six admitted members hold opinions $0,5,2,3,2$ and $1$. The locations of $y_k$ (above) and $y_k^+$ (below) are depicted by a small disk, and the intervals which are mapped by $\psi^+_k$ (above) and $\psi_k$ (below) to each integer are written next to it. Observe that $m_k$ could be computed by rounding down either $y_k$ (or $y^+_k$) to the nearest integer, then rounding up (down) to the nearest element of the image, and then looking on the infimum of the corresponding preimage interval (in accordance with \eqref{eq: m and psi}). Also observe how $y_k$ and $y_k^+$
evolve in exactly the same way as long as $x_k\neq m_k$ (see \eqref{eq: y def}), and
$\psi^+_k,\psi_k$ evolve in the same way except interval endpoints, as long as the admitted members have distinct opinions (see \eqref{eq: diff psi}).}
\label{fig: psi}
\end{figure}

For given $a<b$, write
\begin{align}\label{eq: I}
I_k^{a,b} := (\psi_k(a)-\D_k,\ \psi^+_k(b)], \qquad J_k^{a,b} :=I_k^{a,b}+ \D_k, \qquad
\Delta_k := y_k^+ -y_k\ge 0.
\end{align}
In what follows, we fix such $a<b$ and write $I_k=I^{a,b}_k$ and $J_k=J_k^{a,b}$ for short.
We turn to prove a few short claims concerning these sets.

\begin{clm}\label{clm: in I}
$ y_k^+ \in J_k  \iff y_k \in I_k \iff a\le m_k \le b$.
\end{clm}

\begin{proof}
Using \eqref{eq: m and psi}, we have $y_k^+ \le \psi_k^+(b) +\D_k \iff y_k \le \psi_k^+(b) \iff m_k \le b$. Similarly, we have
$y_k > \psi_k(a)-\D_k \iff y_k^+ > \psi_k(a) \iff m_k\ge a$.
\end{proof}

\begin{clm}\label{clm: I increase}
$|I_{k+1}| - |I_k|  = |J_{k+1}|-|J_k|= \emph{\ind}\{a\le x_{k}\le b\}$. \\ In particular
$|I_k|$ and $|J_k|$ form the same integer-valued non-decreasing sequence, started at $|I_0|=0$.
\end{clm}

\begin{proof}
By \eqref{eq: diff psi}, we have:
\begin{align*}
|I_{k+1}| - |I_k|
&= \psi_{k+1}^+(b) -\psi_k^+(b)  - (\psi_{k+1}(a) - \psi_{k}(a) ) + \D_{k+1}-\D_k
\\& =\ind\{ x_{k}\le b\} - \ind\{x_{k}\le m_{k}\}
- \big(\ind\{ x_{k}<a\} - \ind\{x_{k}<m_{k}\}\big)
\\ & \qquad+ \ind\{x_{k}\le m_{k}\} -\ind\{x_{k} < m_{k}\}
\\ & = \ind\{x_{k}\le b\} - \ind\{x_{k} < a\}
 = \ind\{a\le x_{k}\le b \}.
\end{align*}
Which establishes the claim for $I_k$, and hence for $J_k$, as, by \eqref{eq: I}, $J_k=I_k+\Delta$.
\end{proof}

\begin{prop}\label{prop: Ik=1}
Suppose that $\mu([a,b])>0$.
Then, almost surely, either $m_k > b$ for all sufficiently large $k$, or there exists $\ell\in\N$ such that
$x_\ell\in [a,b]$.
\end{prop}

\begin{proof}
We assume
$\Pro(m_k \le a)\ge \Pro(x_0,x_1,\dots,x_{k-1} \le a)>0$, as otherwise the claim is straightforward. Hence,\vspace{-5pt}
\begin{align}
\Pro( x_{k} \in [a,b] \mid \cF_{k-1}, \ m_k \le a)& \notag
\ge \Pro(\min(X,Y) \in [a,b]\ |\ \tfrac{X+Y}2\ge a)\\ &
\ge \Pro(X,Y \in [a,b]) = \mu( [a,b])^2>0.  \label{eq: low}
\end{align}

By the Borel-Cantelli lemma, this implies that almost surely one of the following holds:
\begin{mitemize}
\item there exists $\ell$ such that $x_{\ell}\in [a,b]$.
\item  $x_{k}\notin [a,b]$ for all $k$ and there exists $\ell_0$ such that for all $\ell>\ell_0$ we have $m_{\ell}>a$.
\end{mitemize}
From Claim~\ref{clm: I increase} it follows that in the former case $|I_\ell|=1$ for some $\ell$, while in the latter case $m_k>b$ for all $k>\ell_0$.
The claim follows.
\end{proof}

In the following proposition we capture the idea that the probability that $m_k$ crosses an interval of negative drift decays exponentially in the number of club members whose opinion lies in that interval.
Recall the notation $\nu_{k}(I) = \left| \{j \le k: \: x_j\in I \}\right|$.

\begin{prop}\label{prop: ell people}
Let $[a,b]$ be an interval of positive measure-$\mu$ such that $\rho(x)<-\ep<0$ for all $x\in [a,b]$. Then there exists $q=q\Big(\ep, \frac{\mu([a,b])}{\mu([a,\infty))} \Big)\in (0,1)$ such that
for all $\ell, t_0\in\N$:
\begin{equation}\label{eq: ell people}
\Pro\Big( \exists k>t_0: \ m_k \ge b \ \Big| \ \cF_{t_0-1}, m_{t_0}\le a, \ \nu_{t_0-1}([a,b])\ge \ell \Big) <q^\ell.
\end{equation}
\end{prop}

\begin{proof}
Denote $I_k=I^{a,b}_k$ as in \eqref{eq: I}. In light of \eqref{eq: m and psi}, Claim \ref{clm: in I} and Claim \ref{clm: I increase} it would suffice to show that,
for all $\ell, t_0\in\N$,
\begin{equation}\label{eq: goal}
\Pro\Big(   \exists k>t_0: \  y_k >\psi_k^+(b) \, \Big| \,   \cF_{t_0-1}, y_{t_0} < I_{t_0}, |I_{t_0}| \ge \ell \Big)<q^\ell.
\end{equation}

Define inductively for $i>0$
\begin{align*}
t_i&=\min\{k>t_{i-1} : \ m_{k-1}<a, \ m_k\ge a\},\\
s_i&=\min\{k\ge t_i :\  \quad a\le m_k\le b, \ m_{k+1}\not \in [a,b] \},
\end{align*}
By Claim \ref{clm: in I}, $t_i$ are times in which the walk $y_k$ entered the interval $I_k$ from below while $s_i$ are times in which $y_k$ is about to exit the interval in its next step.
Denote $p:=\frac{\Pro(X \in [a,b])}{\Pro(X\in [a,\infty))}$ and observe that $p>0$.
By Claim \ref{clm: I increase} and the definition of the quantile process, we almost surely have,
\begin{align}\label{eq: p0}
\Pro\big(|I_{t_i}|>|I_{t_i-1}|\ \big|\ \cF_{t_i-1} \big) & \ge \Pro(x_{t} \in [a,b] \mid \cF_{t-1}, m_{t-1}\le a, m_t\ge a)\notag
\\ & \ge\Pro(\min(X,Y) \in [a,b]\ |\ \tfrac{X+Y}2\ge a)\ge p.
\end{align}

Write $i_{\infty}=\min\{i: \ t_i=\infty\}$.
For each $i<i_\infty$, and $k\in [t_i,s_i]$ write $X_k^{(i)}=y_k -y_{t_i}$ and observe that this is
a random walk started at $X_0^{(i)}=0$ 
and satisfying $X_{k+1}^{(i)} - X_k^{(i)} = y_{k+1} - y_k$ for $k\in [t_i,s_i]$.
By our premise and \eqref{eq: rho def} we have $\E\Big(X_{k+1}^{(i)} \mid X_0^{(i)},\dots,X_k^{(i)}, \cF_{t_i-1} \Big) \le X_k^{(i)}-\ep$ for all $k\in\N_0$. Further observe that $\{X^{(i)}_j\}_{j\in\N,i\in[k]}$ are $\cF_{s_i}$ measurable.
By \eqref{eq: y pm} we have $|X^{(i)}_{k+1}-X^{(i)}_k|\le 1 $ for any $i$ and $k$.
Denote $\ell_j=|I_{t_j}|-|I_{t_{j-1}}|$ for $j\ge 1$ and $\ell_0 =|I_{t_0}|$.
Consequently, for $i\in \N_0$ we have $L_i := \sum_{j\le i} \ell_j = |I_{t_i}|$.
By Claim~\ref{clm: I increase}, $\ell_j\ge 0$ for all $j$ and by \eqref{eq: p0}
we have $\Pro(\ell_j>0 \mid \cF_{t_j-1}) >p$
for all $j$. Since $\ell_j$ is $F_{t_j-1}$ measurable, and since $t_{i+1}\ge s_{i}+1$
we conclude that the random processes $\{X_k^{(i)}\}$ and the process $\{\ell_j\}$ obey the conditions of Lemma
\ref{lem: ruin2} with $\eta=\ep$,
$\cF^{(i)}_0 = \cF_{t_i-1}$ and $\cF^{(i)}_1 = \cF_{s_i}$.
We deduce that
$$\Pro\Big(\exists i,k\in \N\text{ s.t. }X^{(i)}_k \ge L_i\Big) \le q^{\ell },$$
for some $q\in (0,1)$ which depends only on $\ep$ and $p$.

In order to obtain \eqref{eq: goal} it is enough to show that
\begin{equation}\label{eq: we pass the interval}
\forall k\in [t_i,s_i]: \qquad \big\{ X^{(i)}_k \ge L_i \big\} \supseteq \big\{y_k > \psi_k^+(b)\big\}.
\end{equation}
Using the fact that $\{L_i\}$ is integer-valued together with \eqref{eq: I}, we have:
\[
\{ X^{(i)}_k \ge L_i \}=
\{ X^{(i)}_k > L_i-1 \}=
\{ y_k >  y_{t_i}  + |I_{t_i}| -1\}=
\{ y_k > \psi^+_{t_i}(b)-\psi_{t_i}(a)+y^+_{t_i} -1\},\]
so that in order to obtain \eqref{eq: we pass the interval} it would suffice to show that
\begin{equation}\label{eq: for passing the interval}
\psi_k^+ (b) \ge  \psi^+_{t_i}(b)-\psi_{t_i}(a)+y^+_{t_i} -1.
\end{equation}
Recalling that $m_j\le b$ for all $j\in[t_i,s_i]$ and using
\eqref{eq: psi moves} we have $\psi_{j+1}^+(b)\ge \psi_j^+(b)$ for all such $j$, and thus
\begin{equation}\label{eq: for passing 1}
\psi_k^+ (b) \ge  \psi^+_{t_i}(b).
\end{equation}
Next, recalling the definition of $t_i$, we have $m_{t_i-1}< a$, so that by
\eqref{eq: m and psi} we obtain $y^+_{t_i-1}\le \psi_{t_i-1}(a)$.
By \eqref{eq: mk def} we have $x_{t_i-1}>a>m_{t_i-1}$ and hence, by \eqref{eq: diff psi}, we deduce that
$\psi_{t_i}(a) = \psi_{t_i-1}(a)$. 
Putting these observations together we obtain
\begin{equation}\label{eq: for passing 2}
y^+_{t_i} = y^+_{t_i-1} +r \le \psi_{t_i-1}(a) +r =  \psi_{t_i}(a) +r < \psi_{t_i}(a)+1.
\end{equation}
Inequalities \eqref{eq: for passing 1} and \eqref{eq: for passing 2} imply \eqref{eq: for passing the interval}, concluding the proof.

\end{proof}

\begin{prop}\label{prop: ell people rev}
Let $[b,a]$ be an interval with $\mu([b,a])>0$, such that $\rho^+(x)>\ep>0$ for all $x\in [b,a]$. Then there exists $q=q(\ep, \mu([b,a]))$ such that
for all $\ell, t_0\in\N_0$:
\begin{equation}\label{eq: ell people rev}
\Pro\Big( \exists k>t_0: \ m_k \le b \ \Big| \ \cF_{t_0-1}, m_{t_0}\ge a, \ \nu_{t_0-1}( [b,a])\ge \ell \Big) <q^\ell.
\end{equation}
\end{prop}

\begin{proof}
Under appropriate changes, such as replacing $\rho$ with $\rho^+$ and the intervals $I_k$ with $J_k$ (defined in \eqref{eq: I}), most of
the proof is identical to that of Proposition \ref{prop: ell people}. The only significant difference is in the proof of a counterpart of \eqref{eq: p0}, namely:
\begin{equation}\label{eq: p0 for J}
\Pro(|J_{t}|> |J_{t-1}|\mid\cF_{t-1},m_{t-1}>a, m_t\le a ) \ge \frac{\mu([b,a])}{\mu((-\infty,a])} \ge \mu([b,a]).
\end{equation}
Thus, we provide here only the proof of \eqref{eq: p0 for J}. To this end,
we observe, using \eqref{eq: mk def},
that the event $\{m_{t-1}>a, \ m_t\le a\}$
occurs if the median of $x_0,\dots, x_{t-2}$ is above~$a$ (i.e. $\nu_{t-2}((-\infty,a])\ge r(t-1)$),
the next accepted member's opinion lies in $(-\infty,a]$ (i.e. $x_{t-1}\le a$), and
this causes the new median $m_t$ to be less then $a$ (i.e. $\nu_{t-2}((-\infty,a])+1<rt$).
We thus have:
\begin{align*}
\Pro(|J_{t}|> &|J_{t-1}|\mid\cF_{t-1}, m_{t-1}>a, m_t\le a )
\\ & = \Pro(x_{t-1} \in [b,a] \mid \cF_{t-1}, m_{t-1}> a, m_t\le a)
\\ & = \Pro(x_{t-1} \in [b,a] \ |\ m_{t-1}, r(t-1)\le  \nu_{t-2}((-\infty,a]) < rt-1 , x_{t-1}\le a ),\\
&=\Pro\left(\min(X,Y) \in [b,a]\ |\  m_{t-1}, m_{t-1}>a, \tfrac{X+Y}2\ge  m_{t-1},  \min(X,Y)\le a\right)
\end{align*}
The first equality follows from Claim \ref{clm: I increase}.
In the last equality we have used the fact that the distribution of $x_{t-1}$ is equal to that of $\min(X,Y)$ conditioned on $\tfrac{X+Y}{2}\ge m_{t-1}$,
where $X$ and $Y$ are i.i.d. $\mu$-distributed random variables.
We then observe that for any fixed values $m>a$ and $y\ge m$ we have
\begin{align*}
\Pro(\min(X,Y) &\in [b,a]\ |\ m_{t-1}=m>a, \tfrac{X+Y}2\ge  m,  \min(X,Y) \in [0,a], Y=y)\\
&=\Pro\left(X \in [b,a]\ |\ X\in [2m-y, a]\right)\\
&\ge  \Pro\left(X \in [b,a]\ |\  X \in [0,a]\right) = \frac{\mu([b,a]}{\mu((-\infty,a])}.
\end{align*}
Therefore \eqref{eq: p0 for J} holds, as required.

\end{proof}

We conclude this part
with the following simple observation about the continuity of the process $m_k$, following directly from \eqref{eq: mk def} and from the fact that members are admitted one-by-one.
\begin{obs}\label{obs: m_k contin.}
Let $i< j<\ell$ and suppose that $m_j<x_i<m_\ell$. Then there exists $k\in (j,\ell)$ such that $m_{k}=x_i$.
\end{obs}

\subsection{A property of continuous measures}\label{sec: meas}

We shall use the following observation about continuous measures.

\begin{lem}\label{lem: two good intervals}
Let $\mu$ be probability measure with no atoms in $[a,b]$ which satisfies $\mu([a,b])>0$. Then
there exist two intervals $I_1, I_2\subseteq[a,b]$, both of positive measure,
such that $\frac{I_1+I_2}{2} > I_1$.
\end{lem}

\begin{proof}
Let $a'=\inf\{d\ge a :\ \forall \epsilon>0 \: \mu((d,d+\epsilon))>0\}$
and $b'=\sup\{d\le b :\ \forall \epsilon>0 \: \mu((d-\epsilon,d))>0\}$.
Since $\mu([a,b])>0$ these infimum and supremum are taken over non-empty sets and are therefore finite.
By the continuity of $\mu$ on $[a,b]$ we have $a'< b'$.
Fix $\delta = \frac{b'-a'}{4}$ and set $I_1=(a',a'+\delta)$ and $I_2=(b'-\delta,b')$. Then $\mu(I_j)>0$ for $j=1,2$ and
$\inf(I_1)+\inf(I_2) = a'+b'-\delta
=2(a'+\delta)+\frac{(b'-a')}4
>2(a'+\delta)\ge 2\sup(I_1).$

%
%

%
\end{proof}

\section{Proofs of key lemmata}\label{sec: lems}
\subsection{Proof of Lemma \ref{lem: right separators}: negative drift}\label{sec: drift}
 \textbf{Proof outline.}
The proofs of the three parts of the lemma are rather similar. All three rely on the observation that to the left of every point at which 
$\rho(m_k)$, the drift of $y_k$,
is negative,
one can identify an interval $I$ of positive measure where the drift is bounded away from zero from below uniformly.
Our purpose is to use this property to show that if $m_k$ is below $I$ for sufficiently many steps, then it will never leave it from above.
This is shown in Lemma~\ref{lem: aux right} below. To show Lemma~\ref{lem: aux right}, we divide the process $m_k$ into excursions into $I$ entering it from below. We show that the probability of $m_k$ ever exiting $I$ from above in a given excursion is less than $1$, and that it decays exponentially in the number of club members whose opinion lies in $I$ at the start of the excursion (using Proposition~\ref{prop: ell people}). We further show that $m_k$ will exit $I$ from below after sufficiently many visits to $I$ (using Lemma~\ref{lem: ruin} on drifting random walks, and our choice of $I$), and that each time that a new excursion starts, there is a positive probability, bounded away from zero, that a new member with opinion in $I$ will be admitted to the club. We then argue that the probability of ever exiting $I$ from above decays to $0$ exponentially fast as the number of members in $I$ accumulates, so that, by the Borel-Cantelli Lemma, the probability that infinitely many such crossings will occur must be zero.

\bigskip

{Our key auxilary lemma is the following}.
\begin{lem}\label{lem: aux right}
let $\ep>0$, $\al \le \beta \le \g$ and $I\in\{(\beta,\gamma],(\beta,\gamma)\}$, such that $\mu( [\al,\beta]) >0$,
$\mu((\al,\beta]\cup I)\le \frac{\ep}{2}$ and $\rho(x)\le -\ep$ for all $x\in [\alpha,\beta]\cup I$.
Then $\g$ is a right-barrier, and if $I\neq \emptyset$ then $I$ is a right-separator.
\end{lem}

{We present the proof of Lemma~\ref{lem: right separators} and then turn to prove Lemma \ref{lem: aux right}.}

\begin{proof}[Proof of Lemma~\ref{lem: right separators}]
{\bf Part (\ref{lem: right separators: item plus}).}
Let $x_0$ be such that $\rho^+(x_0)<0$ and write $\ep=-\frac{\rho^+(x_0)}2$.
Let $$a:=\inf\left\{x\ :\ \mu(x,x_0)\le \frac{\ep}2, \rho^+(x)\le -\ep \right\}.$$
Using Observations~\ref{obs: rho monotone for mu null} and \ref{obs: lim at -infty}, we deduce that $a>-\infty$ and
$\mu([a,x_0))>0$. Using the fact that $\rho\le \rho^+$ (Observation \ref{obs: rho contin}), we deduce that $\rho(x)\le -\ep$ for all $x\in[a,x_0)$.
As $\mu([a,x_0))>0$, there exists $b\in [a,x_0)$ such that $\mu([a,b]) >0$.
Hence $a\le b< x_0$ and $\ep$ satisfy the conditions of Lemma~\ref{lem: aux right}. We conclude that $(b,x_0]$ is a right-separator, as required.

{\bf Part (\ref{lem: right separators: item barrier}).}
Let $x_0$ be such that $\rho(x_0)<0$.
If $\rho^+(x_0)<0$ then this part follows from part~\eqref{lem: right separators: item plus}. Otherwise, we have $\rho^+(x_0)\ge 0$ so that, by Observation \ref{obs: rho contin}, we have $\mu(\{x_0\})>0$. Hence the conditions of Lemma~\ref{lem: aux right} are satisfied with $\al=\beta=\g=x_0$ and $\ep=-\rho(x_0)$. We conclude that $x_0$ is a right-barrier, as required.

{\bf Part (\ref{lem: right separators: item right limit}).}
Let $x_0$ be such that $\lim_{x\to x_0+} \rho(x)<0$ and write
$\ep=-\lim_{x\to x_0+} \rho(x)/2$. Write
\begin{align*}
b&=\sup\left\{x :\ \mu(x_0,x)\le \frac{\ep}4,\, \rho(x)\le -\ep \right\},\\
a&=\inf\left\{x :\ \mu(x,b)\le \frac{\ep}2, \,\rho(x)\le -\ep \right\}.
\end{align*}
Observe that $a\le x_0<b$. By lower-semicontinuity of $\rho$ (Observation~\ref{obs: rho contin}), we have
$\rho(x)\le -\ep$ for all $x\in [a,b]$. By Observation~\ref{obs: rho monotone for mu null}, we see that $\mu([a,x_0])>0$. Hence the conditions of
Lemma~\ref{lem: aux right} are satisfied with $\al=a$, $\beta=x_0$, $\g=b$ and $\ep$. We deduce that $(x_0,b]$ is a right-separator, so that, by definition,  $(x_0,b)$ is also a right-separator -- as required.
\end{proof}

\begin{proof}[Proof of Lemma~\ref{lem: aux right}]
Denote the events $E=\{m_k> I \text{  a.e.}\}$ and $T=\{m_k< I \text{  a.e.}\}$ (here $E$ is associated with \emph{escaping} while $T$ is associated with being \emph{trapped}). Our purpose is to show that $\Pro(E\cup T)=1$.

We define a sequence of discrete stopping times $\{t_i\}_{i\in\N_0}$, setting
$t_0=\inf\{t\in\N\ :\ x_t\in[\alpha,\beta] \}$ and defining for $i>0$,
$$t_i=\begin{cases}
\inf\{t>t_{i-1}\ :\ m_t\le \alpha,\ m_{t-1}> \alpha\} & t_{i-1}<\infty\\
\infty & t_{i-1}=\infty
\end{cases},$$
these are subsequent entry times of $m_k$ into the set $(-\infty,\alpha]$.
From Proposition~\ref{prop: Ik=1} we obtain that, almost surely, $\{t_0<\infty\}\cup E$ holds.

Next, we use induction to show that
\begin{equation}
\Pro(\{\forall{i\in\N_0}:t_i<\infty\}\cup E\cup T)=1.
\label{eq: A B or finitely many t i less than infty}
\end{equation}
Indeed, assume that $t_{i}<\infty$ and let us show that either $t_{i+1}<\infty$, $E$ or $T$ holds.
To this end, define a sequence of stopping times $\{s_j^i\}_{\{j\in\N_0\}}$ by setting $s_0^i=\inf\{s>t_{i}\ :\ m_s\in
(\alpha,\beta]\cup I\}$ and
$$s_j^i=\begin{cases}\inf\{s^i_{j-1}<s<t_{i+1} :\ m_s\in (\alpha,\beta]\cup I, \ m_{s-1}\notin (\alpha,\beta]\cup I\} & s^i_{j-1}<\infty\\
\infty & s^i_{j-1}=\infty.
\end{cases}$$
These are subsequent entry times of $m_k$ to the set $(\alpha,\beta]\cup I$ between times $t_i$ and $t_{i+1}$. To show~\eqref{eq: A B or finitely many t i less than infty}, it would suffice to show that $\Pro(\{\forall j\in\N_0:s^i_j<\infty\})=0$.
Define $Z_k := y_{k} - \psi^+_k(\al)$ and observe that, by \eqref{eq: m and psi}, $t_{i+1}=\inf\{k>s_0^i\ :\ Z_k\le0\}$.
Writing $A_k=\{m_{k}\in(\alpha,\beta]\cup I\}$, we compute
\begin{align*}
\E (Z_{k+1}-Z_k &\mid \cF_{k-1}, A_k) = \E(y_{k+1} -y_{k} \mid \cF_{k-1}, A_k ) - \E (\psi_{k+1}^+(\al) -\psi_{k}^+(\al)\mid \cF_{k-1},A_k )
\\ & \le -\ep
+\Pro (x_{k}\le m_{k} \mid \cF_{k-1}, A_k)- \Pro (x_{k}\le \al \mid \cF_{k-1}, A_k)\\
&\le -\ep +\Pro(x_{k} \in (\al,\beta]\cup I \mid \cF_{k-1}, A_k) \le-\ep+ \mu( (\al,\beta]\cup I ) \le-\frac{\ep}{2},
\end{align*}
where the transition between the first line and the second uses \eqref{eq: rho def}, \eqref{eq: diff psi} and our assumption on $\rho$
and the last inequality uses our assumption that $\mu((\alpha, \beta]\cup I)\le \frac{\ep}2$. Write
$$\tau^i_j=\inf\{k>s_j^i\ :\ m_k>\gamma\} =\inf\{k>s_j^i\ :\ Z_k >\psi_k^+(\g)-\psi_k^+(\al)\} $$ and
observe that
$\{t_{i+1}<\tau^i_j\}\subseteq \{s_{j+1}^i=\infty$\}.
Notice that the process $\{Z_k: k\in [s_j^i, \min(\tau_j^i, t_{i+1})\}$ has bounded steps and drift at most $-\ep$,
so by applying Lemma~\ref{lem: ruin},
there exists $q>0$ such that
\begin{align*}
\Pro(s^i_{j+1}=\infty\ |\ \cF_{s_j^i-1}, s^i_j<\infty)
&\ge \Pro(t_{i+1}<\tau^i_j\ |\ \cF_{s_j^i-1}, s^i_j<\infty)
\\ & \ge \Pro(T_{\{Z_k \ge 1\}}\ |\ Z_0\le 0)
>q.
\end{align*}
Thus, for all $j\in \N_0$ such that $\Pro(s^i_{j}<\infty)>0$ we have $\Pro(s^i_{j+1} <\infty \mid s^i_{j} <\infty)<1-q$. Using the conditional second Borel-Cantelli lemma \cite{Brus} we obtain $\Pro(\{\forall j\in\N_0:s^i_j<\infty\})=0$, and establish~\eqref{eq: A B or finitely many t i less than infty}.

Next, applying Proposition~\ref{prop: ell people}
with $[a,b]:=[\al,\beta]$, $t_0:=t_i$, and $\ep$, we obtain the existence of $q<1$ such that for all $i\in\N_0$,
\begin{equation}\label{eq: B beta large}
\Pro(T^c\ |\ \cF_{t_i-1}, t_i<\infty)< \Pro\big(\exists k>t_i:\  m_k\ge \beta \ \big| \ \cF_{t_i-1}, m_{t_i}\le \al, 
\nu_{t_i-1}([\al,\beta] )\ge 1\big) < q.
\end{equation}
Denote $L_j=|\{i<j\ :\ \exists t\in (t_i,t_{i+1}), m_t>\beta\}|$ for $j\in\N\cup\{\infty\}$. By \eqref{eq: B beta large} for all $i\in\N$ we obtain, $\Pro(L_\infty>L_i\ |\ \cF_{t_i-1}, t_i<\infty)< q$.
Hence $\Pro(\{L_\infty>k\}\cap \{\forall{i\in\N_0}:t_i<\infty\})<q^k$ for any $k\in\N$, so that
$\Pro(T^c\cap \{\forall{i\in\N_0}:t_i<\infty\})=0$.
Combined with \eqref{eq: A B or finitely many t i less than infty} it follows that $\Pro(E\cup T)=1$, so that by definition, if $I\neq \emptyset$ then
$(\beta,\g]$ is a separator, while if $I=\emptyset$, then $\g$ is a barrier.
Moreover, from \eqref{eq: B beta large} we obtain $\Pro(T)>0$, so that if $I\neq \emptyset$, then $(\beta,\g]$ is a right-separator, while if $I=\emptyset$, then $\g$ is a right-barrier.
\end{proof}

\subsection{Proof of Lemma \ref{lem: left separators}: positive drift}

The proof is analogous to that of Lemma~\ref{lem: right separators}. Firstly, we introduce a counterpart of Lemma~\ref{lem: aux right}.
\begin{lem}\label{lem: aux left}
let $\ep>0$ and $\g \le \beta \le \al$ be such that $\mu((\g,\al))\le \frac{\ep}{2}$, $\mu( [\beta,\al]) >0$ and
$\rho^+(x)\ge \ep$ for all $x \in I\cup[\beta,\alpha]$ for $I\in\{[\gamma,\beta),(\gamma,\beta)\}$.
Then $\Pro(m_k> I\text{ a.e.})>0$ and if $I\neq \emptyset$ then $I$ is almost surely a separator, while if $I=\emptyset$ then, almost surely, either $m_k \ge \gamma$ almost everywhere or $m_k<\gamma$ almost everywhere.
\end{lem}
\emph{Mutatis mutandis}, the proof of Lemma \ref{lem: aux left} is the same as that of Lemma~\ref{lem: aux right}. A few noteworthy alterations are the following. Define $T=\{m_k> I \text{  a.e.}\}$ and $E=\{m_k< I \text{ a.e.}\}$ and $Z_k := \psi_k(\al)-y^+_{k}$. In addition, Proposition~\ref{prop: ell people rev} plays in the proof the role previously given to Proposition~\ref{prop: ell people}.

\begin{proof}[Proof of Lemma~\ref{lem: left separators}]

{\bf Part (\ref{lem: left separators: item plus}).}
Let $x_0$ be such that $\rho(x_0)>0$ and denote $\ep:=\frac{\rho(x_0)}2.$
Write
\begin{equation}\label{eq: set a}
a=\sup\left\{x\ :\ \mu(x_0,x)\le \frac{\ep}2, \rho(x)\ge \ep \right\}.
\end{equation}
Using Observation~\ref{obs: rho monotone for mu null} and the fact that $x_0<\mu_{\max}$ we deduce that
$\mu((x_0,a])>0$. Using the fact that $\rho\le\rho^+$ (Observation \ref{obs: rho contin}) and
that $\rho^+(x) = \lim_{t\to x-}\rho(t)$, we deduce that $\rho^+(x)\ge \ep$ for all $x\in[c,a]$
for some $c<x_0$.
As $\mu((x_0,a])>0$, we can choose $b\in (x_0,a]$ such that $\mu([b,a]) >0$.
Hence $c<b\le a$ and $\ep$ satisfy the conditions of Lemma~\ref{lem: aux left}. We conclude that $({c},b)$ is a left-separator
{containing $x_0$}, as required.

{\bf Part (\ref{lem: left separators: item barrier}).}
Let $x_0$ be such that $\rho^+(x_0)>0$.
If $\rho(x_0)>0$ then this part follows from part~\eqref{lem: right separators: item plus}. Otherwise, we have $\rho(x_0)\le 0$. If $\rho(x_0)<0$, then by Lemma~\ref{lem: right separators}\ref{lem: right separators: item barrier}, $x_0$ is a right-barrier and we are done. In the remaining
$\rho(x_0)=0$ case, by Observation \ref{obs: rho contin}, we have $\mu(\{x_0\})>0$. Hence the conditions of Lemma~\ref{lem: aux left} are satisfied with $\al=\beta=\g=x_0$ and $\ep=\rho^+(x_0)$.
We conclude that either $m_k\ge x_0$ almost everywhere, or $m_k< x_0$ almost everywhere. By Observation~\ref{obs: rho jumps at atom}, we have
$\lim_{x\to x_0+}\rho(x)>0$.
Hence there exists $b>x_0$ such that $\rho^+(x)\ge \ep>0$ for all $x\in [x_0,b]$ and some $\ep>0$. By Lemma \ref{lem: aux left}, $(x_0,b]$ is a separator and hence $x_0$ is a barrier.

{\bf Part (\ref{lem: left separators: item right limit}).}
Let $x_0$ be such that $\lim_{x\to x_0+} \rho(x)>0$ and denote
$\ep:=\lim_{x\to x_0+} \rho(x)/2$.
We repeat the arguments of Part \ref{lem: left separators: item plus}, setting $a$ as in \eqref{eq: set a} and $b\in (x_0,a]$ so that $\mu([b,a])>0$. This way we obtain that $\rho^+(x)\ge \ep$ for all $x\in(x_0,a]$.
Hence $x_0<b\le a$ and $\ep$ satisfy the conditions of Lemma~\ref{lem: aux left}  and we deduce that, almost surely, $(x_0,b)$ is a separator, as required.

\end{proof}

\subsection{Proof of Lemma \ref{lem: zero drift barrier}: zero drift }\label{sec: no drift}
\textbf{Proof outline.}
Our purpose is to show that if a process visits a positive measure interval $I$ with $\rho(I)=0$, sufficiently many times, then it will eventually find itself entrapped within. To do so we identify inside $I$ a sub-interval $I_1$, which we in turn sub-divide into a countable sequence of nested shells with outer shells thinner than inner ones. We then show that after every visit of $m_k$ to the innermost shells, there is a probability bounded away from zero that it will never leave the interval. To do so we construct an event under which this happens, and bound its probability from below. On this event, by the time at which the walk breaks out from the $i$-th shell, many elements are added to the $i+1$-th shell, so that it takes significantly more time for the walk to cross it. By showing that this sequence of times converges to infinity we conclude that the walk never breaks out of $I_1$.

\bigskip

Let $I=[a,b]$ be such that $\mu(I)>0$ and $\rho^+(x) = \rho(x) = 0$ for all $x\in I$.
By Observation \ref{obs: rho contin}
this implies that $\mu$ has no atoms in $[a,b]$.
Hence, by Lemma \ref{lem: two good intervals}, there exist intervals $I_1,I_2\subseteq I$, such that
$\mu(I_j)>0$ ($j=1,2$) and $\frac{I_1+I_2}{2}> I_1$.
We further divide the interval $I_1 = [\al,\beta]$ as follows. Let $c\in I_1$ be the minimal point satisfying
$\mu([\al,c])=\mu([c,\beta])$.
Define two sequences $c=r_0<r_1<r_2<\dots< \beta$ and $c=\ell_0>\ell_1>\ell_2>\dots>\al$ such that
$\mu([r_j, r_{j+1}]) = \mu([\ell_{j+1},\ell_j]) =  2^{-(j+2)} \mu(I_1)$.
Fix a parameter $n_0\in \N$.
For any $t>0$ and $j\in\N$ let
\[
 N_j(t) =   \min\Big(\nu_{t-1} ( [\ell_{j+2},\ell_{j+1}]), \nu_{t-1}([r_{j+1},r_{j+2}])\Big).
 \]
Fix $t_0, n_0\in \N$ and denote the event $B_{0}=\{m_{t_0}\in [\ell_1,r_1], \, N_0(t_0)\ge n_0\}$.
Our proof relies on the following lemma, whose proof appears later in this section.
 \begin{lem}\label{lem: stuck forever}
 There exists $n_0$ such that if $t_0>\frac{n_0}{r(1-r)}+5$, then $\Pro(B_0)>0$ and
 \[
 \Pro(\forall k>t_0: \, m_k \in I_1 \mid \cF_{t_0-1}, B_{0})>\frac 1 3.
 \]
 \end{lem}
We turn to establish Lemma \ref{lem: zero drift barrier}.

\begin{proof}[Proof of Lemma \ref{lem: zero drift barrier}]
The fact that $\Pro( m_k\in I\, \text{a.e.})>0$ is immediately implied by Lemma \ref{lem: stuck forever}.
We are left with showing that almost surely, $I$ contains a barrier.

Let $F :=\{m_k<\ell_1  \text{ i.o.}\}\cap  \{ m_k > r_1 \text{ i.o.}\} $.
If $F^c$ holds then either $r_1$ or $\ell_1$ is a barrier and we are done. 
Thus it remains to show that
\begin{equation}\label{eq: remains}
\Pro(F \cap \{ I \text{  does not contain a barrier}  \} ) = 0.
\end{equation}
If $\Pro(F)=0$ we are done. Otherwise, let $n_0$ be as in Lemma \ref{lem: stuck forever}.
For $j\in \N$ define
\begin{align*}
s_0 &=\min\{s\in\N: \: m_{s}\in [\ell_1,r_1] ,N_0(s)>n_0\},
\\ s_j &= \min\{s>s_{j-1}: \: m_{s}\in [\ell_1,r_1], \: m_{s-1}\not\in[\ell_1,r_1]\}.
\end{align*}
Under $F$, we have $s_j<\infty$ for all $j$ (using Proposition~\ref{prop: Ik=1} with $[a,b]=[\ell_1,r_1]$ and Observation \ref{obs: m_k contin.}).
By Lemma~\ref{lem: stuck forever} we have
$\Pro( \{\forall k\ge s_j:\ m_k\in I_1\ |\ \cF_{s_j-1}\} \cap F)>\frac 1 3,$ for all $j\in\N$.
Thus by the conditional Borel-Cantelli Lemma \cite{Brus},
$\Pro( \{m_k\in I_1 \, \text{a.e.} \} \cap F)=1$, which implies that under $F$, every point of $I_2\subset I$ is a barrier,
\eqref{eq: remains} follows.
\end{proof}

It remains to prove Lemma \ref{lem: stuck forever}.
For $j\in\N$, define inductively
\vspace{-5pt}
\begin{align*}
&t_j := \inf\{k >t_{j-1}: \:  m_k\not\in [\ell_j,r_j]\},
 \quad N_j :=
\begin{cases}
N_j(t_j), & t_j<\infty,\\
\infty, &t_j =\infty.
\end{cases}
\\[-25pt]
\end{align*}
Let $\delta=2^{-3}\mu(I_1)\mu(I_2)$ and denote (for $j\in\N$) the events:
\vspace{-5pt}
\begin{align*}
A_j = \{ t_{j+1}-t_j \ge N_j^{1.5}\}\cup \{t_{j}=\infty\}, \quad B_j = \{N_{j+1} \ge  \delta 2^{-j}N_j^{1.5}\}\cup \{t_{j}=\infty\}.\\[-25pt]
\end{align*}
Notice that both $A_j$ and $B_j$ are $\cF_{t_{j+1}-1}$-measurable.
In proving Lemma \ref{lem: stuck forever} we rely on three key properties, whose proofs we provide at the end of this section.

\begin{clm}\label{clm: A}
For each $j\in\N$:
\[
\Pro(A_j^c \mid \cF_{t_j-1}) \le e^{- \sqrt{N_j}/ 2}.
\]
\end{clm}

\begin{clm}\label{clm: H}
There exists a constant $C>0$ such that for all $j\in \N$,
\[
\Pro\left( B_j \mid \cF_{t_j-1},  A_j \right)
\ge 1 - 2 e^{-C 2^{-2j} N_j^{1.5}}.
\]
\end{clm}

\begin{clm}\label{clm: Q}
For all $Q>0$ and $t_0>0$ there exists $n_0=n_0(Q)$ such that
\[
\bigcap_{j\in \N_0} B_j \subseteq
\bigcap_{j\in\N} \{N_j \ge e^{Q (1.1)^j}\}.
\]
\end{clm}

\begin{proof}[Proof of Lemma \ref{lem: stuck forever}]
First we show that for any $n_0$ and $t_0>\frac{n_0}{r(1-r)}+5$ we have $\Pro(B_0)>0$.
Let $K_1, K_2, K_3\subset \N$ be disjoint sets of integers such that
$\{0,1,\dots, t_0-1\} = K_1\cup K_2\cup K_3$,
$|K_1| =  \lceil \frac{n_0}{1-r} \rceil$,
$|K_2| = \lceil \frac{n_0}{r}\rceil$, and $|K_3|\ge 3$.
Let $(X_k,Y_k)$ denote the opinions of the independent candidates at stage $k$. Define
\vspace{-13pt}
\begin{align*}
G_0:= &\ \{\forall k\in K_1: \: X_{k}\in [\ell_2,\ell_1]\}  \cap
 \{\forall k\in K_2: \: X_{k}\in [r_1,r_2]\} \\ & \cap
\{ \forall k\in K_3: \: X_{k}\in [\ell_1,r_1]\} \cap
\{\forall k\le t_0: \: Y_{k}\in I_2 \}.
\end{align*}
Under $G_0$ we have $N_0(t_0)\ge \min\Big( \frac{n_0}{r},\frac{n_0}{1-r}\Big)\ge n_0$. Now notice that under $G_0$ we have
$$\nu_{t_0-1}( [\ell_2,\ell_1])\!=\!|K_1|\! <\! r(|K_1|+|K_2|+|K_3|)\!=\!r t_0\text{ and }
\nu_{t_0-1}([\ell_2, r_1])\!=\!|K_1|+|K_2|\! >\! r (|K_1|+|K_2|+|K_3|)\!=\!r t_0,$$
which implies that $m_{t_0} \in [\ell_1,r_1]$. We deduce that $G_0\subseteq B_0$.
Since $\{X_k\}$ and $\{Y_k\}$ are all independent it follows that $\Pro(G_0)\ge \Pro(B_0)>0$.

Let $Q>0$ be a large parameter, to be chosen later.
By Claim \ref{clm: Q}, there exists $n_0=n_0(Q)$ so that
\begin{equation}\label{eq: clm Q}
\bigcap_{j\in \N_0}B_j\subseteq \bigcap_{j\in\N} \{ N_j \ge e^{Q (1.1)^j}\}.
\end{equation}
Let $t_0>\frac{n_0}{r(1-r)}+5$, so that $\Pro(B_0)>0$ by the arguments above. Using \eqref{eq: clm Q} and the bounds from 
Claims~\ref{clm: A} and~\ref{clm: H} we have:
\begin{align*}
\Pro&\Big( \forall j: \ N_j \ge   e^{Q (1.1)^j} \mid \cF_{t_0-1}, B_0\Big)
\ge \Pro\Big(\bigcap_{j\in\N} B_j \bigcap_{j\in\N} A_j \mid \cF_{t_0-1}, B_0\Big)
\\ &= \prod_{j\in \N} \Pro(B_j  \mid \cF_{t_0-1},B_0,\dots,B_{j-1}, A_1,\dots,A_j) \prod_{j\in \N} \Pro(A_j \mid \cF_{t_0-1}, B_0,\dots,B_{j-1}, A_1,\dots,A_{j-1})
 \\ &  \ge \prod_{j=1}^\infty \left(1-2e^{- C 2^{-2j} N_j^{1.5}}\right) \left(1-e^{-\sqrt{N_j}/2} \right)
\ge \prod_{j=1}^\infty \left(1-2e^{-C 4^{-j}e^{1.5 Q(1.1)^j }} \right)\left(1-e^{- \frac 1 2 e^{\frac 1 2Q(1.1)^j }}\right)
\\ &
  \ge 1 - 2 \sum_{j=1}^\infty e^{- C 4^{-j}e^{1.5 Q(1.1)^j } } - \sum_{j=1}^\infty e^{-\frac 1 2 e^{\frac 1 2Q(1.1)^j }}
 \end{align*}
We may choose $Q$ sufficiently large to have both
$C e^{1.5 Q (1.1)^j}\ge 8^j$ and $\frac  1 2 e^{\frac 1 2 Q(1.1)^j}\ge 2^j$ for all $j\in\N$.
For such $Q$ we obtain
\[
\Pro( \forall j: \ N_j > e^{Q (1.1)^j} \mid B_0)
\ge 1 - 2 \sum_{j=1}^\infty e^{- C 4^{-j}e^{1.5 Q(1.1)^j } } - \sum_{j=1}^\infty e^{-\frac 1 2 e^{\frac 1 2Q(1.1)^j }}
\ge  1 - 3\sum_{j=1}^\infty e^{-2^j}\ge   \frac 1 3.
\]
Since $t_{j+1} \ge N_j$, we have $\{\lim_{j\to\infty} t_j =\infty\} \supseteq \{\forall j: \: N_j\ge e^{Q (1.1)^j}\}$. Thus
 \[
\Pro( m_k\in I_1\: \forall k\ge t_0
\mid \cF_{t_0-1},  B_0 )\ge  \Pro ( \lim_{j\to\infty} t_j= \infty\mid \cF_{t_0-1}, B_0) \ge \frac 1 3,
\]
as required.
\end{proof}

At last, we prove Claims \ref{clm: A}, \ref{clm: H} and \ref{clm: Q}.

\smallskip
\begin{proof}[Proof of Claim \ref{clm: A}]

We observe that, since $\rho^+\equiv 0$ on $I_1$, the process $\{y_k^+-y_{t_j}^+: \ k\in [t_{j},t_{j+1}]\}$ is
a martingale $Z_k$ whose steps are uniformly bounded by $1$, started at $0$
and stopped when $k=t_{j+1}$, that is, when $m_k\not \in [\ell_j,r_j]$.
Observe that at this stopping time $|Z_{t_{j+1}}|=|y_{t_{j+1}}^+ - y_{t_j}^+| \ge N_j$. We may thus apply Azuma's inequality (Lemma~\ref{lem: A}) to obtain
\[
\Pro (t_{j+1}-t_{j}< N_j^{1.5}, t_j<\infty \mid \cF_{t_j-1} ) \le \Pro(T_{\{X_k >N_j\}} <N_j^{1.5} \mid \cF_{t_j-1}) \le \exp\big(-{N_j^2}/{(2N_j^{1.5})} \big) = e^{-\sqrt{N_j}/2}.
\]

\end{proof}

\smallskip
\begin{proof}[Proof of Claim \ref{clm: H}]

We begin by noticing that for every $j,k\ge 0$ we have
\vspace{-5pt}
\begin{align}\label{eq: more people}
\Pro(x_{k} \in  [r_{j+1},r_{j+2}] \mid m_{k}\in I_1, \cF_{k-1} )
& \ge \Pro( \{ X\in  [r_{j+1},r_{j+2}], \ Y\in I_2 \} \cup \{Y\in  [r_{j+1},r_{j+2}],\   X\in I_2\}) \notag
\\ &= 2 \, \mu( [r_{j+1},r_{j+2}] ) \, \mu(I_2) =  2^{-(j+2)}\mu(I_1)\ \mu(I_2).
\end{align}
The same holds when $[r_j,r_{j+1}]$ is replaced by $[\ell_{j+1},\ell_j]$.

Now, assume that $t_j<\infty$ and
denote
$N_{j+1}' :=\nu_{t_{j+1}-1} ([r_{j+1},r_{j+2}])$
and  $N_{j+1}'':= \nu_{t_{j+1}-1} ([\ell_{j+2},\ell_{j+1}])$
so that $N_{j+1} = \min(N_{j+1}', N_{j+1}'')$.
Let $Z\sim \text{Bin}(n,p)$ be a binomially distributed random variable with
$n=N^{1.5}_j$ and $p=2^{-j-2}\mu(I_1)\mu(I_2)$.
By \eqref{eq: more people}, the random variables $N_{j+1}'$ and $N_{j+1}''$, conditioned on  $t_{j+1}-t_j \ge N^{1.5}_j$ and $\cF_{t_j-1}$, stochastically dominate~$Z$. 
Thus,
\begin{align*}
&\Pro\left( N'_{j+1} > 2^{-j-3} \mu(I_1)\ \mu(I_2)\ N_j^{1.5} \mid \cF_{t_j-1}, \ t_{j+1}-t_{j}\ge N_j^{1.5} \right) \ge \Pro\left(Z >  \frac{pn}2  \right), 
\end{align*}
and the same bound holds when $N'_{j+1}$ is replaced by $N''_{j+1}$.
Since $N_{j+1} = \min(N_{j+1}', N_{j+1}'')$ we may apply a union bound to obtain
\begin{align*}
\Pro( B_j& \mid  \cF_{t_j-1}, t_{j}<\infty, A_j)
\\ &\ge  1-\Pro\left( N'_{j+1} \le \delta 2^{-j} \ N_j^{1.5} \mid \cF_{t_j-1}, \ A_j \right)
-\Pro\left( N'_{j+1} \le  \delta 2^{-j} \ N_j^{1.5} \mid \cF_{t_j-1}, \ A_j \right) 
\\ & \ge 1-2\Pro(Z>\frac{pn}2) 
\\ &\ge  1 - 2e^{-p^2n/2}= 1-2e^{-C 2^{-2j} N_j^{1.5}},
\end{align*}
where in the last line we applied Hoeffding's bound (Lemma \ref{lem: H}) with $\ep =p/2$, and $C=2^{-5}\mu(I_1)^2\mu(I_2)^2$.
\end{proof}

\smallskip
\begin{proof}[Proof of Claim \ref{clm: Q}]
Setting $X_j=\log N_j$ for $j\in \N_0$, we have:
\begin{align*}
\bigcap_{j\in\N_0}\!\!B_j & \subseteq \bigcap_{j\in\N_0}\!\!\left\{N_{j+1} \ge \delta2^{-j}N_j^{1.5}\right\}\cap
\{N_0 \ge n_0\}  \subseteq \bigcap_{j\in \N_0}\!\!\left\{X_{j+1} \ge \frac 3 2  X_j  -  j\log 2 +\log \delta\right\}\cap  \{X_0 \ge \log n_0\} .
\end{align*}
Using induction one sees that, for any $Q>0$, under the right-most event, $X_j\ge Q(1.1)^j$ for all $j\ge 1$, provided that $n_0$ is sufficiently large.
The claim follows.

\end{proof}

\subsection{Proof of Lemma~\ref{lem:as bounded}: boundedness}\label{sec: bound}

We may assume that $\mu$ is not compactly supported, since otherwise the lemma is straightforward.
By Observation~\ref{obs: lim at -infty}, there exists $M$ such that $\rho^+(m) > \frac r 2$ for all $m\in (-\infty, M]$
and $\mu((-\infty,M])\in (0,\frac r 4)$. Hence, by Lemma~\ref{lem: aux left} applied with $\gamma=-\infty$, $\beta=-\infty$ and $\al=M$, the sequence of quantiles $m_k$ almost surely enters $(-\infty,M)$ only a finite number of times,
 so that $m_k$ is almost surely bounded from below.

By Theorem \ref{thm: FF} and the definition of $\rho^+$,
$\liminf_{m\to\infty}\rho^+(m) =-(1-r)$, so that there exists an increasing sequence $b_j\to\infty$ such that, for all $j\in \N$,
\begin{equation}\label{eq: to use soon}
\rho^+(b_j) < -\frac{7}{8}(1-r),
\end{equation}
and $\mu([b_1,\infty) )< \frac 1 4$.
Let $a_j$ be a sequence be such that $a_j\le b_j<a_{j+1}$ and
\begin{equation}\label{eq: abj}
\frac{1}{4} \le \frac{\mu([a_j,b_j))} {\mu([a_j,\infty) )}, \qquad \frac{\mu((a_j,b_j))} {\mu((a_j,\infty) )} \le \frac 1 2.
\end{equation}
Then for each $x\in [a_j,b_j]$ we have (using \eqref{eq: to use soon} and the right side of \eqref{eq: abj}):
\begin{equation*}
\frac{\mu((x,\infty) )^2}{(\mu*\mu)([2 x,\infty) )}=
\frac{\left[\mu ((x,b)) +\mu([b,\infty)) \right]^2 }{(\mu*\mu)([2 x,\infty) )}
\le \frac{4 \mu([b_j,\infty))^2}{(\mu*\mu)([2 b_j,\infty) )} = 4(\rho^+(b_j)+1-r) < \frac{1-r}{2}.
\end{equation*}
Therefore,
\begin{equation}\label{eq: traps}
\forall j\in \N: \quad \rho|_{[a_j,b_j]} < -\frac{1-r}{2}.
\end{equation}

Let $t_0=n_0=1$. For each $j\in \N$ define inductively
\[
t_j := \inf\{k>t_{j-1}:\  m_k > b_{n_{j-1}} \}, \qquad k_j =\min\{k\in\N: \  m_{t_j}< b_k\}.
\]
Our goal is to prove that, almost surely, there exists a finite $j$ with $t_j=\infty$.
Define for $j\in\N$:
\[
s_j :=\inf\{k\in\N:\  x_k\ge a_{n_j}\}.
\]
We have $s_j \le t_j$. By \eqref{eq: abj} we have
\[
\Pro(x_{s_j} \in [a_{n_j},b_{n_j}] \mid s_j<\infty, \cF_{s_j-1} )\ge \frac 1 {16},
\]
which yields
\[
\Pro\big( \nu_{t_j-1}( [a_{n_j},b_{n_j}]) \ge 1 \mid t_j<\infty, \cF_{t_j-1}\big) \ge \frac 1 {16}.
\]

By Proposition~\ref{prop: ell people} applied with $a=a_{n_j}$, $b=b_{n_j}$, $\ep=\frac{1-r}{2}$ we have
\[
\Pro\left( t_j = \infty \mid t_{j-1}<\infty, \cF_{t_j-1},  \nu_{t_j-1}([a_{n_j},b_{n_j}])\ge 1 \right)>1-q.
\]
Here $q$ depends only on the negative upper bound on $\rho|_{[a_{n_j},b_{n_j}]}$, which, by \eqref{eq: traps} is $-\frac{1-r}{2}$ and thus uniform in $j$,
 and the lower bound on $\frac{\mu([a_{n_j},b_{n_j}])}{\mu([a_{n_j},\infty))}$, which by the left part of \eqref{eq: abj} is also uniform in $j$.
 We conclude that for all $j$,
 \[
 \Pro\left( t_j = \infty \mid t_{j-1}<\infty, \cF_{t_j-1}\right) \ge \frac 1 {16}(1-q)>0.
 \]
  Hence we may apply the conditional Borel-Cantelli lemma, and obtain that almost surely there exists a finite $j$ for which $t_j=\infty$, as required.


\section{Uniqueness of the limit}\label{sec: exms}

\subsection{Proof of Proposition \ref{prop: atom exm}:
an example of non-uniqueness}
Fix a parameter $p\in (0,1)$.
Let $\mu$ be the atomic probability measure defined by
$$\mu= \sum_{\ell\in\N_0}(1-p)p^\ell\delta_{1-2^{-\ell}},$$
where $\delta_x$ is the Dirac measure at $x$.
Observe that for every $a\in \{1-2^{-\ell}\}$:
\begin{equation}\label{eq: ex atom prop}
\Pro\left(\frac{X+Y}{2}\ge a \right)=\Pro\left(X\ge a\right)\Pro\left( Y\ge a\right),
\end{equation}
where $X, Y$ are i.i.d. $\mu$-distributed random variables.
Let $r\in(0,1)$ be such that $r<1-p^2$. At every atom $a = 1-2^{-\ell}$ we have
\begin{align*}
\rho^+(a) &=\frac{\Pro(X\ge a)^2}{\Pro\left(\frac{X+Y}{2}\ge a\right)}-1+r = r>0, \\
\rho(a) &= \frac{\Pro(X> a)^2}{\Pro\left(\frac{X+Y}{2}\ge a\right)} - 1+r = \left(\frac{\Pro(X>a)}{\Pro(X\ge a)}\right)^2-1+r
=p^2-1+r<0.
\end{align*}
Thus, by Proposition \ref{prop: ell people} applied with $[a,b]=\{a\}$, which is of positive measure and negative drift,
for any $k_0\in\N$ we have
\[
\Pro(m_k \le a \text{ for all } k\ge k_0 \mid \cF_{k_0-1}, m_{k_0}\le a, \nu_{k_0-1}(\{a\})\ge 1) >0.
\]
This implies that the stopping time $\tau := \min\{t\in \N: \, m_t=a\}$ satisfies
\[
\Pro(m_k \le a \text{ for all } k\ge \tau  \mid \cF_{\tau-1}, \tau<\infty ) >0.
\]
We note that $\Pro(\tau<\infty)>0$ (e.g. $\tau=1$ under the event that both candidates in the first round of applications hold the opinion $a$, which is an event of positive probability).
In addition, by \eqref{eq: ex atom prop}, $\Pro(\exists k>\tau: \, m_k <a \mid \cF_{\tau-1}, \tau<\infty)=0$, so that
$
\Pro(m_k  = a \text{ for all } k\ge \tau  \mid \cF_{\tau-1}, \tau<\infty ) >0
$. Hence
\[
\Pro(m_k = a \text{ a.e.}  ) >0.
\]
Thus we have shown that $\Pro(\nu_\infty=\delta_a)>0$ for every atom $a$ of $\mu$ -- and hence that $\nu_\infty$ is supported on an infinite number of limit distributions, as required.

\subsection{Proof of Proposition \ref{prop: convex exm}: a condition for uniqueness}

Our proof uses ideas from \cite{FF}.
Write $\Pro(X\ge m) = e^{-g(m)}$. By the hypothesis, $g''\ge 0$ and $g'''\le 0$ everywhere in $[\mu_{\min},\mu_{\max}]$.
We note that, in particular, $g'$ must be a concave, non-decreasing, positive function on $[\mu_{\min},\mu_{\max}]$.
As these conditions cannot hold when $\mu_{\min}=-\infty$, we may assume, without loss of generality, that $\mu_{\min}=0$.
By \eqref{eq: rho} we have $\rho(m) = \frac{1}{\Psi(m)} +r-1$ where
\[
\Psi(m)= \frac{\Pro\left(\frac{X+Y}{2}\ge m \right) }{\Pro(X\ge m)^2}.
\]
We need to prove that $\Psi(m)$ is strictly monotone increasing in $m$. We have:
\[
\Pro(X+Y\ge 2m) =  \int_{0}^\infty \Pro(Y \ge  2m-x) d\mu(x)  = \int_0^{\infty} e^{-g(2m-x)} d\mu(x)
\]
However $d\mu = -d(e^{-g}) = e^{-g} g'$, so that:
\begin{align*}
\Psi(m) &= \frac{\int_0^{\infty} e^{-g(2m-x) -g(x)} g'(x) dx }
{e^{-2g(m)} }
\\ &= \int_0^{\infty} e^{2g(m)-g(x)-g(2m-x)} g'(x) dx  
\\ & = \int_0^{m} e^{2g(m)-g(m+x)-g(m-x)} \big(g'(m+x)+g'(m-x)\big) dx + e^{2g(m)-g(2m)}
\end{align*}

Differentiation with respect to $m$ yields,
\begin{align*}
\Psi'(m) =&\int_0^{m}e^{2g(m)-g(m+x)-g(m-x)} \\
& \:\: \times\Big( \left(2g'(m)-g'(m+x)-g'(m-x)\right)\left(g'(m+x)+g'(m-x)\right) +g''(m+x)+g''(m-x)\Big) dx
\\ &+ e^{2g(m)-g(2m)} (2g'(m)+2g'(2m)),
\end{align*}

by the premises that $g'\ge 0$ and $g''\ge 0$ on the support of $\mu$.
Moreover, for $x<m$ we have\\
$2g'(m)-g'(m+x)-g'(m-x) = -2x g'''(c_{m,x}) > 0$ (for some $c_{m,x}\in (m-x,m+x)$).
Lastly, since $\mu_{\min}=0$, we have $g'(0)>0$.
We conclude that $\Psi'(m)\ge 4e^{2g(m)-g(2m)} g'(0) >0$ for all $m\in [\mu_{\min},\mu_{\max}]$, so that $\Psi$ is strictly monotone
on the support of $\mu$, as required.

\bigskip

\subsection*{Acknowledgement}
We are grateful to Noga Alon for introducing to the problem and for useful discussions.

\end{document}